\documentclass{article}%

\usepackage{amsmath,amssymb,amsfonts}%
\usepackage{theorem}%
\usepackage{color}%
\usepackage{hyperref}%

\def\MR#1#2{\href{http://www.ams.org/mathscinet-getitem?mr=#1}{#2}}%

\theoremstyle{change}%
\newtheorem{definition}{Definition:}[section]%
\newtheorem{proposition}[definition]{Proposition:}%
\newtheorem{theorem}[definition]{Theorem:}%
\newtheorem{lemma}[definition]{Lemma:}%
\newtheorem{vlemma}[definition]{Volume lemma:}
\newtheorem{corollary}[definition]{Corollary:}%
{\theorembodyfont{\rmfamily}\newtheorem{remark}[definition]{Remark:}}%
{\theorembodyfont{\rmfamily}}%

\newenvironment{proof}
  {{\bf Proof:}}
  {\qquad \hspace*{\fill} $\Box$}%

\newcommand{\RT}{\operatorname{Re}}%
\newcommand{\ad}{\operatorname{ad}}%
\newcommand{\id}{\operatorname{id}}
\newcommand{\inner}{\operatorname{int}}%
\newcommand{\cl}{\operatorname{cl}}%
\newcommand{\tm}{\times}%
\newcommand{\ep}{\varepsilon}%
\newcommand{\inv}{\operatorname{inv}}%
\newcommand{\spn}{\operatorname{span}}%
\newcommand{\vol}{\operatorname{vol}}%
\newcommand{\Lip}{\operatorname{Lip}}%
\newcommand{\Graph}{\operatorname{Graph}}%
\newcommand{\diam}{\operatorname{diam}}%
\newcommand{\Mo}{\operatorname{Mo}}%
\newcommand{\Ly}{\operatorname{Ly}}%

\newcommand{\rmd}{\mathrm{d}}%
\newcommand{\rme}{\mathrm{e}}%

\newcommand{\AC}{\mathcal{A}}%
\newcommand{\CC}{\mathcal{C}}%
\newcommand{\EC}{\mathcal{E}}%
\newcommand{\FC}{\mathcal{F}}%
\newcommand{\LC}{\mathcal{L}}%
\newcommand{\OC}{\mathcal{O}}%
\newcommand{\QC}{\mathcal{Q}}%
\newcommand{\SC}{\mathcal{S}}%
\newcommand{\UC}{\mathcal{U}}%
\newcommand{\VC}{\mathcal{V}}%
\newcommand{\N}{\mathbb{N}}%
\newcommand{\R}{\mathbb{R}}%
\newcommand{\Z}{\mathbb{Z}}%

\begin{document}

\title{Invariance Entropy of Hyperbolic Control Sets\footnote{2010 \emph{Mathematics Subject Classification}. Primary: 93C15, 37D20, 37C60}}%
\author{Adriano Da Silva\footnote{Imecc - Unicamp, Departamento de Matem\'atica, Rua S\'ergio Buarque de Holanda, 651, Cidade Universit\'aria Zeferino Vaz 13083-859, Campinas - SP, Brasil; ajsilvamat@hotmail.com}\ \ and Christoph Kawan\footnote{Courant Institute of Mathematical Sciences, New York University, 251 Mercer Street, New York, N.Y. 10012-1185, USA; kawan@cims.nyu.edu}}%
\maketitle%

\begin{abstract}%
In this paper, we improve the known estimates for the invariance entropy of a nonlinear control system. For sets of complete approximate controllability we derive an upper bound in terms of Lyapunov exponents and for uniformly hyperbolic sets we obtain a similar lower bound. Both estimates can be applied to hyperbolic chain control sets, and we prove that under mild assumptions they can be merged into a formula.%
\end{abstract}

{\small {\bf Keywords:} Invariance entropy; Control sets; Chain control sets; Hyperbolicity; Control-affine systems; Volume lemma; Universally regular controls; Shadowing lemma}

\section{Introduction}%

Invariance entropy is a measure for the smallest rate of information above which a control system is able to render a given compact controlled invariant subset of the state space invariant. This concept, first introduced in Colonius and Kawan \cite{CKa}, is essentially equivalent to the topological feedback entropy, defined by Nair, Evans, Mareels and Moran \cite{Nea}. A comprehensive treatment of the subject can be found in the monograph \cite{Ka2}. If the given set is a control set, i.e., a maximal set of complete approximate controllability, upper bounds of the invariance entropy in terms of the sums of positive Lyapunov exponents of periodic solutions can be given. Under a completely different assumption, namely the existence of a uniformly hyperbolic structure (in a skew-product sense) on the controlled invariant set, there exists a promising approach for an optimal lower estimate in terms of the unstable determinant. Specific examples of compact controlled invariant sets are the bounded chain control sets of a control-affine system. These sets are the projections of the maximal chain transitive sets of the associated control flow, which is a skew-product on the extended state space, including the shift dynamics on the set of admissible control functions. Moreover, a general result in Colonius and Du \cite{CDu} shows that under the assumption of local accessibility a chain control with nonempty interior and a uniformly hyperbolic structure is the closure of a control set. Hence, in this case good entropy estimates from above and from below are available. In this paper, we improve these estimates and merge them into a formula. The paper consists of three main sections whose contents are briefly described as follows.%

In Section \ref{sec_gub}, we improve the known upper bounds for the entropy of a control set. In particular, we show that the hyperbolicity assumption imposed in \cite[Sec.~5.2]{Ka2} can be dropped without substitution. Then we obtain an upper bound in terms of the Lyapunov exponents of the induced linear system on the exterior bundle of the state space. This estimate has some similarity with the integral formula for the topological entropy of $\CC^{\infty}$-maps, established in Kozlovski \cite{Koz} (based on previous work of Yomdin \cite{Yom} and others). However, instead of an integral, i.e., an average of the exponential growth rates, the infimum over the growth rates has to be considered. An essential ingredient in the proof of this estimate is a result of Coron \cite{Cor} which implies that under a strong accessibility assumption the set of universally regular control functions is generic inside the set of smooth control functions.%

In Section \ref{sec_glb}, we obtain a lower estimate for the entropy of a uniformly hyperbolic controlled invariant set. The main ingredient of the proof is a skew-product version of the Bowen-Ruelle volume lemma. This lemma was first formulated by Liu \cite{Liu} in the context of random dynamical systems. Liu, however, only gives an outline of the proof with many details missing. Following this outline, we develop a fully detailed proof in the context of general continuous skew-products with compact base space (without use of the control structure present in the context of control flows). Combining the volume lemma with ideas from Young \cite{You} for the estimation of escape rates, we derive a lower estimate of the invariance entropy that is similar to the upper estimate of Section \ref{sec_gub}.%

Finally, in Section \ref{sec_form} we provide a formula for the entropy of a hyperbolic chain control set, showing that the upper and lower bounds of the preceding sections coincide. More precisely, we show that%
\begin{equation*}
  h_{\inv}(Q) = \inf_{(u,x)}\limsup_{\tau\rightarrow\infty}\frac{1}{\tau}\log\left|\det(\rmd\varphi_{\tau,u})|_{E^+_{u,x}}\right|,%
\end{equation*}
where $Q$ is the hyperbolic chain control set and the infimum is taken over all pairs $(u,x)$ of control functions and states such that the corresponding trajectory $\varphi(t,x,u)$ remains in $Q$ for all times $t\in\R$. The linear subspace $E^+_{u,x}$ is the corresponding fiber of the unstable subbundle. The extra work needed to obtain this formula mainly consists in proving an approximation result for periodic points and a periodic shadowing property for the shift flow on the set of admissible control functions. This enables us to derive the above formula for smooth control-affine systems under the mild assumption that the Lie algebra rank condition is satisfied on the interior of the chain control set.%

\section{Preliminaries}\label{sec_prelim}

\subsection{Notation}%

We write $\Z$, $\N$, $\R$ and $\R_+$ for the sets of integers, positive integers, real numbers and nonnegative real numbers, respectively, and $\R^d$ for the $d$-dimensional Euclidean space. Moreover, $\N_0 = \N \cup \{0\}$. If $V$ is a finite-dimensional real vector space, $V^*$ denotes its dual space, the space of real-valued linear functionals on $V$. By a smooth manifold we understand a finite-dimensional connected second-countable Hausdorff manifold endowed with a $\CC^{\infty}$-differentiable structure. If $M$ is a smooth manifold, we denote by $T_xM$ the tangent space at $x\in M$, by $0_x$ (or simply $0$) the zero element of $T_xM$, and by $TM$ the tangent bundle. If $\varphi:M\rightarrow N$ is a differentiable map between smooth manifolds, we write $(\rmd\varphi)_x:T_xM \rightarrow T_{\varphi(x)}N$ for its derivative at $x\in M$. If $(M,g)$ is a Riemannian manifold, we write $\langle\cdot,\cdot\rangle$ for the inner product and $|\cdot|$ for the induced norm on each tangent space, while $\|\cdot\|$ is used for operator norms. We write $\varrho$ for the induced distance function on $M$ and $\vol$ for the Riemannian volume measure. By $\exp_x$ we denote the Riemannian exponential map at $x\in M$. Furthermore, we write $\cl A$ for the topological closure of a set, and $\inner A$ for its interior. The open ball of radius $\ep>0$ at $x\in M$ is denoted by $B(x,\ep)$. The abbreviation ``a.e.'' stands for ``(Lebesgue-) almost everywhere''. If $x$ is a real number, $[x]$ denotes the greatest integer $\leq x$. We write $\log^+x = \max\{0,\log x\}$, and we let $\#S$ denote the number of elements of a finite set $S$. $\chi_A$ stands for the characteristic function of a set $A$.%

\subsection{Control-Affine Systems}%

A control-affine system is given by a family 
\begin{equation}\label{eq_cas}
  \dot{x}(t) = f_0(x(t)) + \sum_{i=1}^mu_i(t)f_i(x(t)),\quad u\in\UC,%
\end{equation}
of ordinary differential equations on a smooth manifold $M$, the \emph{state space} of the system. Here $f_0,f_1,\ldots,f_m$ are $\CC^k$-vector fields for some $k\geq1$. The set $\UC$ of admissible control functions is given by%
\begin{equation*}
  \UC = \left\{u:\R\rightarrow\R^m\ :\ u \mbox{ is measurable with } u(t) \in U \mbox{ a.e.}\right\},%
\end{equation*}
where $U\subset\R^m$ is a compact and convex set. (Frequently, we will also assume $0 \in \inner U \neq \emptyset$.) Then $\UC$, endowed with the weak$^*$-topology of $L^{\infty}(\R,\R^m) = L^1(\R,\R^m)^*$, is a compact metrizable space and the shift flow%
\begin{equation*}
  \theta:\R \tm \UC \rightarrow \UC,\quad (t,u) \mapsto \theta_tu = u(\cdot + t),%
\end{equation*}
is a continuous dynamical system, which is chain transitive. We write $\varphi(\cdot,x,u)$ for the unique solution of \eqref{eq_cas} with $\varphi(0,x,u) = x$. For simplicity, we assume that all solutions are defined on the whole time axis. Then we obtain a map%
\begin{equation*}
  \varphi:\R \tm M \tm \UC \rightarrow M,\quad (t,x,u) \mapsto \varphi(t,x,u),%
\end{equation*}
called the \emph{transition map} of the system, and this map is continuous as well. Together with the shift flow it constitutes a skew-product flow%
\begin{equation*}
  \phi:\R \tm \UC \tm M \rightarrow \UC \tm M,\quad (t,u,x) \mapsto \phi_t(u,x) = (\theta_tu,\varphi(t,x,u)),%
\end{equation*}
called the \emph{control flow} of the system. We also use the notation $\varphi_{t,u}:M\rightarrow M$ for the map $x\mapsto\varphi(t,x,u)$. If the vector fields $f_0,f_1,\ldots,f_m$ are of class $\CC^k$, then $\varphi$ is of class $\CC^k$ with respect to the state variable and the corresponding partial derivatives of order $1$ up to $k$ depend continuously on $(t,x,u) \in \R\tm M\tm\UC$ (cf.~\cite[Thm.~1.1]{Ka2}).%

We define the \emph{set of points reachable from $x\in M$ at time $\tau\geq0$}, the \emph{set of points reachable from $x$ up to time $\tau$}, and the \emph{positive orbit of $x$}, respectively, by%
\begin{eqnarray*}
  \OC^+_{\tau}(x) &:=& \left\{\varphi(\tau,x,u)\ :\ u\in\UC\right\},\\%
  \OC^+_{\leq\tau}(x) &:=& \bigcup_{t\in[0,\tau]}\OC^+_t(x) \mbox{\quad and\quad } \OC^+(x) := \bigcup_{\tau\geq0}\OC^+_{\tau}(x).%
\end{eqnarray*}

In the following, we fix a metric $d$ on $M$ (not necessarily a Riemannian distance). A set $D\subset M$ is called \emph{controlled invariant (in forward time)} if for each $x\in D$ there exists $u\in\UC$ with $\varphi(\R_+,x,u)\subset D$. It is called a \emph{control set} if it satisfies the following properties:%
\begin{enumerate}
\item[(A)] $D$ is controlled invariant.%
\item[(B)] Approximate controllability holds on $D$, i.e., $D\subset\cl\OC^+(x)$ for all $x\in D$.%
\item[(C)] $D$ is maximal (w.r.t.~set inclusion) with the properties (A) and (B).%
\end{enumerate}
From the maximality it easily follows that a control set with nonempty interior has the \emph{no-return property}, i.e., $x \in D$ and $\varphi(\tau,x,u) \in D$ for some $\tau>0$ and $u\in\UC$ implies $\varphi([0,\tau],x,u)\subset D$.%

A set $E\subset M$ is called \emph{full-time controlled invariant} if for each $x\in E$ there exists $u\in\UC$ with $\varphi(\R,x,u)\subset E$. The \emph{full-time lift} $\EC$ of $E$ is defined by%
\begin{equation*}
  \EC = \mathrm{Lift}(E) := \left\{ (u,x) \in \UC \tm M\ :\ \varphi(\R,x,u) \subset E \right\},%
\end{equation*}
which is easily seen to be compact and $\phi$-invariant. For points $x,y\in M$ and numbers $\ep,\tau>0$, a \emph{controlled $(\ep,\tau)$-chain from $x$ to $y$} is given by $n\in\N$, points $x_0,\ldots,x_n\in M$, control functions $u_0,\ldots,u_{n-1}\in\UC$, and times $t_0,\ldots,t_{n-1}\geq\tau$ such that $x_0 = x$, $x_n = y$, and $d(\varphi(t_i,x_i,u_i),x_{i+1}) < \ep$ for $i=0,\ldots,n-1$. A set $E\subset M$ is called a \emph{chain control set} if it satisfies the following properties:%
\begin{enumerate}
\item[(A)] $E$ is full-time controlled invariant%
\item[(B)] For all $x,y\in E$ and $\ep,\tau>0$ there exists an $(\ep,\tau)$-chain from $x$ to $y$ in $M$.%
\item[(C)] $E$ is maximal (w.r.t.~set inclusion) with the properties (A) and (B).%
\end{enumerate}
Every chain control set is closed, which is not the case for control sets. Moreover, every control set with nonempty interior is contained in a chain control set if local accessibility holds, and the full-time lift of a chain control set is a maximal invariant chain transitive subset for $\phi$. Conversely, the projection of a maximal chain transitive set to $M$ is a chain control set.%

System \eqref{eq_cas} is called \emph{locally accessible at $x$} provided that for all $\tau>0$ the sets $\OC^+_{\leq\tau}(x)$ and $\OC^-_{\leq\tau}(x)$ have nonempty interiors. It is called \emph{locally accessible} if it is locally accessible at every point $x\in M$. If the vector fields $f_0,f_1,\ldots,f_m$ are of class $\CC^{\infty}$, the \emph{Lie algebra rank condition} (Krener's criterion) guarantees local accessibility: Let $\LC=\LC(f_0,f_1,\ldots,f_m)$ denote the smallest Lie algebra of vector fields on $M$ containing $f_0,f_1,\ldots,f_m$. If $\LC(x) := \{f(x) : f\in\LC\} = T_xM$, then the system is locally accessible at $x$. If $f_0,f_1,\ldots,f_m$ are analytic vector fields, the criterion is also necessary.%

The concept of \emph{invariance entropy} is defined as follows. A pair $(K,Q)$ of subsets of $M$ is called \emph{admissible} if $K$ is compact and for every $x\in K$ there is $u\in\UC$ with $\varphi(\R_+,x,u) \subset Q$. In particular, if $K=Q$, this means that $Q$ is a compact and controlled invariant set. For $\tau>0$, a set $\SC\subset\UC$ is $(\tau,K,Q)$-spanning if for every $x\in K$ there is $u\in\SC$ with $\varphi([0,\tau],x,u) \subset Q$. Then $r_{\inv}(\tau,K,Q)$ denotes the number of elements in a minimal such set and we put $r_{\inv}(\tau,K,Q) := \infty$ if no finite $(\tau,K,Q)$-spanning set exists. The \emph{invariance entropy} of $(K,Q)$ is%
\begin{equation*}
  h_{\inv}(K,Q) := \limsup_{\tau\rightarrow\infty}\frac{1}{\tau}\log r_{\inv}(\tau,K,Q),%
\end{equation*}
where $\log$ is the natural logarithm. In the case $K=Q$, we also write $r_{\inv}(\tau,Q)$ and $h_{\inv}(Q)$. In this case, the upper limit is in fact a limit, as a consequence of subadditivity. In general, $h_{\inv}(K,Q)$ need not be finite. The number $h_{\inv}(Q)$, however, is finite iff $r_{\inv}(\tau,Q)$ is finite for one or, equivalently, for all $\tau>0$. We refer to \cite{Ka2} for further properties.%

\section{Upper Bounds for Control Sets}\label{sec_gub}%

In this section, we provide a general upper bound for the invariance entropy of a control set in terms of exponential growth rates of the induced system on the exterior bundle of the state space, similar to the well-known integral formula for the topological entropy of $\CC^{\infty}$-maps proved by Kozlovski \cite{Koz}.%

We consider a control system on a smooth manifold $M$ of the general form%
\begin{equation}\label{eq_gencs}
  \Sigma:\quad \dot{x}(t) = F(x(t),u(t)),\quad u \in \UC,%
\end{equation}
with a continuously differentiable right-hand side $F:M\tm\R^m \rightarrow TM$, such that $F_u := F(\cdot,u)$ is a vector field on $M$ for each $u\in\R^m$. The set $\UC$ of admissible control functions is given by%
\begin{equation*}
  \UC = \left\{u:\R\rightarrow\R^m\ :\ u \mbox{ is measurable with } u(t) \in U \mbox{ a.e.}\right\},%
\end{equation*}
where $U\subset\R^m$ is compact with $\inner U\neq\emptyset$. In particular, this includes the control-affine case of the preceding section. We assume that all solutions are defined on the whole time axis and thus we obtain a transition map $\varphi:\R\tm M\tm\UC\rightarrow M$, continuously differentiable in the second component. We also write $\varphi_{t,u}:M\rightarrow M$ for the diffeomorphism $\varphi(t,\cdot,u)$. Control sets and invariance entropy for systems of the general form \eqref{eq_gencs} are defined in the same way as for control-affine systems. The control flow $\phi_t:\UC\tm M \rightarrow \UC\tm M$ is still defined, though it may not be continuous w.r.t.~the weak$^*$-topology on $\UC$. Throughout this section, we assume that $M$ is endowed with a Riemannian metric.%

\subsection{A Ruelle-Pesin-Type Upper Bound}\label{subsec_rptub}

A pair $(\varphi(\cdot,x,u),u(\cdot))$ of a trajectory and the corresponding control function is called a \emph{controlled trajectory}. If the linearization of $\Sigma$ along $(\varphi(\cdot,x,u),u(\cdot))$ is controllable on a time interval $[\tau_1,\tau_2]$ with $\tau_1<\tau_2$, the controlled trajectory is called \emph{regular on $[\tau_1,\tau_2]$}. The linearization is the time-varying linear system%
\begin{equation*}
  \dot{z}(t) = A(t)z(t) + B(t)v(t),\quad v\in L^{\infty}(\R,\R^m),%
\end{equation*}
with $A(t) = (\partial F/\partial x)(\varphi(t,x,u),u(t))$ and $B(t) = (\partial F/\partial u)(\varphi(t,x,u),u(t))$ (written in local coordinates). A periodic controlled trajectory of period $\tau_*$ is called \emph{regular} if it is regular on $[0,\tau_*]$. An easy consequence of this definition is that a trajectory is regular on a time interval iff it is regular on some subinterval. This is made precise in the following proposition (see \cite[Prop.~1.28]{Ka2} for a proof).% 

\begin{proposition}\label{prop_regularity}
A controlled trajectory $(\varphi(\cdot,x,u),u(\cdot))$ is regular on a time interval $[\tau_1,\tau_2]$ iff there exists a non-trivial subinterval $[\rho_1,\rho_2]$, $\tau_1 \leq \rho_1 < \rho_2 \leq \tau_2$, such that $(\varphi(\cdot,x,u),u(\cdot))$ is regular on $[\rho_1,\rho_2]$.%
\end{proposition}

Given a pair $(u,x) \in \UC \tm M$, the \emph{Lyapunov exponents} at $(u,x)$ are the numbers%
\begin{equation*}
  \lambda(v;u,x) := \limsup_{t\rightarrow\infty}\frac{1}{t}\log\left|(\rmd\varphi_{t,u})_x v\right|,\quad v\in T_xM \backslash \{0_x\}.%
\end{equation*}
It is a standard fact that $\lambda(v;u,x)$ can take at most $d=\dim M$ different values. If $(\varphi(\cdot,x,u),u(\cdot))$ is periodic with period $\tau$, the operator $(\rmd\varphi_{\tau,u})_x$ maps $T_xM$ onto itself and the Lyapunov exponents are given by $(1/\tau)\log|\mu|$, $\mu$ being the eigenvalues of $(\rmd\varphi_{\tau,u})_x$. Then the \emph{multiplicity} of a Lyapunov exponent is defined as the sum of the algebraic multiplicities of the corresponding eigenvalues.%

The following lemma characterizes the elements of the interior of $\UC$ w.r.t.~$L^{\infty}(\R,\R^m)$. Its elementary proof can be found in \cite[Prop.~5.4]{Ka2}.%

\begin{lemma}\label{lem_intuchar}
For an element $u\in L^{\infty}(\R,\R^m)$ it holds that $u\in\inner\UC$ iff there exists a compact set $K \subset \inner U$ with $u(t)\in K$ for almost all $t\in\R$.%
\end{lemma}

\begin{theorem}{(\cite[Thm.~4.3]{Ka1}, \cite[Thm.~5.1]{Ka2})}\label{thm_upperboundcs}
Let $D\subset M$ be a control set of $\Sigma$ with nonempty interior and compact closure, and let $(\varphi(\cdot,x,u),u(\cdot))$ be a regular periodic controlled trajectory with $(u,x) \in \inner\UC\tm\inner D$. Then for every compact $K\subset D$ we have%
\begin{equation*}
  h_{\inv}(K,D) \leq \sum_\lambda\max\{0,d_{\lambda}\lambda\},%
\end{equation*}
the sum taken over the Lyapunov exponents $\lambda$ at $(u,x)$ with associated multiplicities $d_{\lambda}$.%
\end{theorem}

As proved in \cite[Prop.~5.11]{Ka2}, the assumptions of regularity and periodicity in this theorem can be weakened under additional assumptions on the given system, involving a weak hyperbolicity assumption. In the following, we give a proof of this result without such an assumption. The general idea is to approximate the Lyapunov exponents of arbitrary trajectories in $D$ by the Lyapunov exponents of regular periodic ones. For this approximation to work, it must be guaranteed that from every $x\in\inner D$ a regular periodic trajectory emanates. A quite general result implying the existence of regular trajectories was proved by Coron \cite[Thm.~1.3]{Cor}. In combination with the complete approximate controllability on $D$ it yields the desired regular periodic orbits.%

\subsection{Existence of Regular Trajectories}\label{subsec_regtraj}%

In the following, we give a brief account of the central definitions in \cite{Cor} in order to explain Coron's theorem. For an open set $V \subset \R^m$, we let $\CC_V^{\infty}(TM)$ be the set of all $\CC^{\infty}$-maps $f:M \tm V \rightarrow TM$ with $f(x,u) \in T_xM$. For $f_1,f_2 \in \CC_V^{\infty}(TM)$ we define the Lie bracket $[f_1,f_2]\in \CC_V^{\infty}(TM)$ by%
\begin{equation*}
  [f_1,f_2](x,u) := [f_1(\cdot,u),f_2(\cdot,u)](x),%
\end{equation*}
where on the right-hand side $[\cdot,\cdot]$ is the usual Lie bracket of vector fields. The \emph{strong jet accessibility algebra} of an element $f\in\CC^{\infty}_V(TM)$ is the linear subspace $\AC=\AC(f)$ of $\CC^{\infty}_V(TM)$ defined as%
\begin{equation*}
  \AC := \spn\left\{\left\{\partial^{|\alpha|}f/\partial u^{\alpha}: \alpha\in\N_0^m, \alpha\neq 0\right\} \cup \mbox{Br}_2\left\{\partial^{|\alpha|}f/\partial u^{\alpha}:\alpha\in\N_0^m\right\}\right\},%
\end{equation*}
where for a family $\FC\subset \CC^{\infty}_V(TM)$, $\mbox{Br}_2(\FC)$ denotes the set of iterated Lie brackets of elements of $\FC$ of length at least two. For example, $\partial f/\partial u^i$, $[f,\partial f/\partial u^i]$, and $\partial^2f/\partial u^i\partial u^j$ are in $\AC$.%

Let $\LC_0$ denote the classical strong accessibility algebra, i.e., the ideal generated by the differences $f(\cdot,u_1) - f(\cdot,u_2)$, $u_1,u_2\in V$, in the Lie algebra generated by the vector fields $f(\cdot,u)$, $u\in V$. Then%
\begin{equation}\label{eq_14}
  \{g(x,u):g\in\AC\} \subset \{g(x):g\in\LC_0\} = \LC_0(x),\ \forall (x,u) \in M \tm V.%
\end{equation}
These inclusions are equalities if, e.g., $f$ is a polynomial w.r.t.~$u$ (including the control-affine case) or if $f$ and $M$ are analytic and $V$ is connected. For $(x,u) \in M\tm V$, let%
\begin{equation*}
  a(x,u) := \{g(x,u):g\in\AC\} \subset T_xM.%
\end{equation*}
If \eqref{eq_14} is an equality, then $a(x,u_1) = a(x,u_2)$ for all $u_1,u_2\in V$. For the convenience of the reader, we give a proof for the control-affine case.%

\begin{proposition}
Assume that $f(x,u) = f_0(x) + \sum_{i=1}^m u_i f_i(x)$ with $\CC^{\infty}$-vector fields $f_0,f_1,\ldots,f_m$. Then%
\begin{equation*}
  a(x,u) = \LC_0(x),\quad \forall (x,u) \in M \tm V.%
\end{equation*}
\end{proposition}

\begin{proof}
The elements of the ideal $\LC_0$ are linear combinations of iterated Lie brackets of the form%
\begin{equation*}
   [X_k,[X_{k-1},[\ldots,[X_1,f_j] \ldots ]]],\quad j \in \{1,\ldots,m\},\ k\in\N_0,%
\end{equation*}
where $X_i \in \{f_0,f_1,\ldots,f_m\}$ (cf.~Nijmeijer and van der Schaft \cite[Prop.~3.20]{NSc}). The partial derivatives $\partial f/\partial u^i$ of first order are the vector fields $f_1,\ldots,f_m$ (regarded as functions on $M \tm V$ rather than $M$). Hence, the higher-order derivatives vanish. The iterated Lie brackets in $\mbox{Br}_2\left\{\partial^{|\alpha|}f/\partial u^{\alpha}:\alpha\in\N_0^m\right\}$ are of the form%
\begin{equation*}
  [X_k,[X_{k-1},[\ldots,[X_1,X_0]\ldots]]], \quad X_i \in \left\{f_1,\ldots,f_m,f_0 + \sum_{i=1}^m u_i f_i\right\},\ k\in\N.%
\end{equation*}
In particular, $a(x,0) = \LC_0(x)$ for all $x\in M$, implying the assertion.%
\end{proof}

Let $x\in M$ and $u$ be a smooth map with values in $V$ defined on a neighborhood of $x$. Let $f_0(y) = f(y,u(y))\in T_yM$, and for $i \in \{1,\ldots,m\}$ let $f_i(y) = (\partial f/\partial u^i)(y,u(y))$. We define $a_l(x;u) \subset T_xM$ by%
\begin{equation*}
  a_l(x;u) := \spn\left\{\ad^k_{f_0}(f_i)(x),\ k\geq0,\ i\in\{1,\ldots,m\}\right\},%
\end{equation*}
where $\ad^0_{f_0}(f_i) = f_i$ and $\ad^k_{f_0}(f_i) = [f_0,\ad^{k-1}_{f_0}(f_i)]$. The subspace $a_l(x;u)$ can be interpreted as follows. Let $\gamma:I\rightarrow M$ be a smooth curve, where $I\subset\R$ is an open interval with $0\in I$ such that%
\begin{equation*}
  \dot{\gamma}(t) = f(\gamma(t),u(\gamma(t))),\quad \gamma(0) = x.%
\end{equation*}
The linearized control system along $\gamma$ is the time-varying linear system%
\begin{equation}\label{eq_timevarcs}
  \dot{z}(t) = A(t)z(t) + B(t)w(t)%
\end{equation}
with%
\begin{equation*}
  A(t) = \frac{\partial f}{\partial x}(\gamma(t),u(\gamma(t))),\quad B(t)w = \sum_{i=1}^m w_i \frac{\partial f}{\partial u_i}(\gamma(t),u(\gamma(t))),%
\end{equation*}
where $w$ is the control and $z(t) \in T_{\gamma(t)}M$ the state. It can be shown that%
\begin{equation}\label{eq_alchar}
  a_l(x;u) = \spn\left\{\left[\left(\frac{\rmd}{\rmd t} - A(t)\right)^iB(t)\right]_{t=0}w;\ w\in\R^m,\ i\geq0\right\}.%
\end{equation}
The right-hand side is the strong accessibility algebra evaluated at $t=0$ of the time-varying linear system \eqref{eq_timevarcs}. In particular, if $a_l(x;u)$ has full dimension, then the linearized system is controllable on every time interval containing $t=0$ (cf.~Sontag \cite[Cor.~3.5.18]{Son}). We say that a control function $u\in \CC^{\infty}(M,V)$ \emph{saturates} $f$ at $x$ if%
\begin{equation*}
  a_l(x;u) = a(x,u(x)).%
\end{equation*}
Moreover, $u$ saturates $f$ on a subset $S\subset M$ if it saturates $f$ at all points of $S$. (We remark that $a_l(x;u) \subset a(x,u(x))$.) Let $Y$ be a smooth manifold, $h\in \CC^{\infty}(M,Y)$, and%
\begin{equation*}
  \Omega \subset \left\{u \in \CC^{\infty}(Y,V) : (\rmd h)_x f(x,u\circ h(x))\neq 0,\forall x\in M\right\}.%
\end{equation*}
Then the result of Coron (\cite[Thm.~1.3]{Cor}) reads as follows.%

\begin{theorem}\label{thm_coron}
Assume that $\Omega$ is open in the $\CC^{\infty}$-topology. Then the set of all $u\in\Omega$ such that $u\circ h$ saturates $f$ on $M$ is the intersection of countably many open and dense subsets of $\Omega$ (in the $\CC^{\infty}$-topology). In particular, this set is dense in $\Omega$.%
\end{theorem}

\begin{remark}
The $\CC^{\infty}$-topology used in the above theorem is finer than the classical Whitney $\CC^{\infty}$-topology. However, as remarked by Coron, the theorem also holds for the Whitney $\CC^{\infty}$-topology.%
\end{remark}

In order to obtain the existence of regular trajectories through every point, we apply the theorem to the extended system on $M\tm\R$, given by%
\begin{equation*}
  \Sigma^*:\ \left\{\begin{array}{rcl} \dot{x}(t) \!\!&=&\!\! f(x(t),u(t))\\ \dot{t} \!\!&=&\!\! 1 \end{array}\right.,\quad u \in L^{\infty}(\R,V).%
\end{equation*}
Putting $Y := \R$, $h(x,t) :\equiv t$,  and $\tilde{f} := (f,1)^T$, we obtain%

\begin{corollary}\label{cor_extendedcs}
The set of $u \in \CC^{\infty}(\R,V)$ such that $u\circ h$ saturates $\tilde{f}$ on $M\tm \R$ is dense in $\CC^{\infty}(\R,V)$.%
\end{corollary}

\begin{proof}
We just need to note that the vector $(\rmd h)_{(x,t)}\tilde{f}((x,t),u(t)) \equiv 1$ never vanishes and hence we can put $\Omega = \CC^{\infty}(\R,V)$ in Theorem \ref{thm_coron}.%
\end{proof}

Now we can deduce the result on the existence of regular trajectories.%

\begin{corollary}\label{cor_regtrajex}
Let $S\subset M$ and assume $a(x,u) = T_xM$ for all $(x,u) \in S \tm V$. Then there is a dense set of $u_0 \in \CC^{\infty}(\R,V)$ such that for every $x\in S$ the controlled trajectory $(\varphi(\cdot,x,u_0),u_0(\cdot))$ is regular on every time interval of the form $[0,\tau]$.%
\end{corollary}

\begin{proof}
Let $a^*((x,t),u)$ and $a_l^*((x,t);u)$ denote the corresponding subspaces of $T_{(x,t)}(M\tm\R) \cong T_xM \tm T_t\R$ for $\Sigma^*$ and note that%
\begin{equation*}
  a^*((x,t),u) = a(x,u) \tm \{0\} = T_xM \tm \{0\}%
\end{equation*}
for all $(x,t,u) \in S\tm\R\tm V$. By Corollary \ref{cor_extendedcs}, there is a dense set in $\CC^{\infty}(\R,V)$ of functions $u_0$ with%
\begin{equation}\label{eq_alsfullrank}
  a_l^*((x,t);u_0 \circ h) = a^*((x,t),u_0(t)) = T_xM \tm \{0\}%
\end{equation}
on $S\tm\R$. Now consider for some $x\in S$ the smooth curve $\gamma(t) := (\varphi(t,x,u_0),t)$ in $M\tm\R$, which satisfies%
\begin{equation*}
  \dot{\gamma}(t) = (f(\gamma(t),u_0(t)),1) = \tilde{f}\left(\gamma(t),u_0 \circ h(\gamma(t))\right),\quad \gamma(0) = (x,0).%
\end{equation*}
In local coordinates, the linearization along $\gamma$ is determined by the matrices%
\begin{eqnarray*}
  \widetilde{A}(t) &=& \frac{\partial\widetilde{f}}{\partial(x,t)}(\gamma(t),u_0(t)) = \left[\begin{array}{cc}
	                                                                                        (\partial f/\partial x)(\varphi(t,x,u_0),u_0(t)) & 0 \\
					                                                                                                                                 0 & 0 \end{array}\right],\\
  \widetilde{B}(t) &=& \frac{\partial\widetilde{f}}{\partial u}(\gamma(t),u_0(t)) =  \left[\begin{array}{c}
	                                                                                        (\partial f/\partial u)(\varphi(t,x,u_0),u_0(t))  \\
					                                                                                                                                 0 \end{array}\right].%
\end{eqnarray*}
With $A(t) := (\partial f/\partial x)(\varphi(t,x,u_0),u_0(t))$ and $B(t) := (\partial f/\partial u)(\varphi(t,x,u_0),u_0(t))$ it follows from \eqref{eq_alsfullrank} and the characterization \eqref{eq_alchar} that%
\begin{equation*}
  \spn\left\{\left[\left(\frac{\rmd}{\rmd t} - A(t)\right)^iB(t)\right]_{t=0}w;\ w\in\R^m,\ i\geq0\right\} = T_xM.%
\end{equation*}
This implies controllability of the linearization along the controlled trajectory $(\varphi(\cdot,x,u_0),u_0(\cdot))$ on every time interval containing $t=0$.%
\end{proof}

\subsection{The Main Result}\label{subsec_gub}%

The next step necessary to generalize Theorem \ref{thm_upperboundcs} is the derivation of another expression for the sum of the positive Lyapunov exponents in the periodic case. This is done in the following proposition.%

\begin{proposition}\label{prop_perposlyapsum}
If $(\varphi(\cdot,x,u),u(\cdot))$ is a periodic controlled trajectory and $\lambda_1(u,x) \geq \cdots \geq \lambda_k(u,x) > 0$ are the positive Lyapunov exponents at $(u,x)$ (counted several times according to their multiplicities), then%
\begin{equation*}
  \lambda_1(u,x) + \cdots + \lambda_k(u,x) = \lim_{t\rightarrow\infty}\frac{1}{t}\log^+\left\|(\rmd\varphi_{t,u})_x^{\wedge}\right\|,%
\end{equation*}
where $(\rmd\varphi_{t,u})_x^{\wedge}:T_x^{\wedge}M \rightarrow T_{\varphi(t,x,u)}^{\wedge}M$ denotes the induced linear operator between the full exterior algebras of $T_xM$ and $T_{\varphi(t,x,u)}M$.%
\end{proposition}

\begin{proof}
We prove the proposition in two steps.%

\emph{Step 1.} The function defined by%
\begin{equation}\label{eq_defcocycle}
  \alpha_t(u,x) := \log^+\left\|(\rmd\varphi_{t,u})_x^{\wedge}\right\|,\quad \alpha:\R\tm(\UC\tm M)\rightarrow\R,%
\end{equation}
is a subadditive cocycle over the control flow on $\UC \tm M$, i.e.,%
\begin{equation*}
  \alpha_{t+s}(u,x) \leq \alpha_t(u,x) + \alpha_s(\phi_t(u,x)),\quad \forall t,s\in\R,\ (u,x)\in\UC\tm M.%
\end{equation*}
This is an easy consequence of the cocycle property of $\varphi$ together with the chain rule and the subadditivity of operator norms. We remark that here neither continuity of the control flow nor of the cocycle $\alpha$ is required. We want to prove the following. For every compact set $K\subset M$ there is a constant $C\geq0$ so that $\varphi(\R_+,x,u)\subset K$ implies%
\begin{equation*}
  \alpha_t(u,x) \leq Ct,\quad \forall t\geq0.%
\end{equation*}
To this end, note that with the operator norm induced by the standard norms on the exterior algebras,%
\begin{equation*}
  \left\|(\rmd\varphi_{t,u})_x^{\wedge}\right\| = \max_{1\leq j\leq d}\sigma_1(t,u,x)\cdot \ldots \cdot \sigma_j(t,u,x),%
\end{equation*}
where $\sigma_1(t,u,x) \geq \cdots \geq \sigma_d(t,u,x)$ are the singular values of $(\rmd\varphi_{t,u})_x:T_xM \rightarrow T_{\varphi(t,x,u)}M$. Hence, we can estimate $\alpha_t(u,x)$ by%
\begin{eqnarray*}
  \alpha_t(u,x) &\leq& \log^+\sigma_1(t,u,x)^d = d\log^+\left\|(\rmd\varphi_{t,u})_x\right\|\\
	&\leq& d\max\left\{0,\int_0^t \lambda_{\max}(S\nabla F_{u(s)}(\varphi(s,x,u))) \rmd s\right\}\\
	&\leq& \max\left\{0,\left[d \max_{(v,z) \in U \tm K}\lambda_{\max}(S\nabla F_v(z))\right]\right\} t =: Ct.% 
\end{eqnarray*}
Here we use that the greatest singular value of an operator is equal to the operator norm, and we write $\lambda_{\max}(S\nabla F_u(\cdot))$ for the maximal eigenvalue of the symmetrized covariant derivative (using Wazweski's inequality, cf.~Boichenko, Leonov and Reitmann \cite{BLR}). This completes the first step.%

\emph{Step 2.} Let $\tau$ denote the period of $(\varphi(\cdot,x,u),u(\cdot))$. By the fundamental lemma of Floquet theory there exists a linear operator $R:T_xM \rightarrow T_xM$ such that%
\begin{equation*}
  (\rmd\varphi_{2\tau n,u})_x = \rme^{2\tau n R},\quad \forall n\in\Z.%
\end{equation*}
Hence, the Lyapunov exponents are the real parts of the eigenvalues of $R$. Writing $t>0$ as $t = 2\tau n(t) + r(t)$ with $n(t) \in \N_0$ and $r(t) \in [0,2\tau)$, we obtain%
\begin{equation*}
   \alpha_t(u,x) \leq \alpha_{2\tau n(t)}(u,x) + \alpha_{r(t)}(u,x),%
\end{equation*}
implying%
\begin{equation*}
  \limsup_{t\rightarrow\infty}\frac{1}{t}\alpha_t(u,x) \leq \limsup_{t\rightarrow\infty}\frac{1}{t}\alpha_{2\tau n(t)}(u,x) = \frac{1}{2\tau}\limsup_{\N\ni n\rightarrow\infty}\frac{1}{n}\alpha_{2\tau n}(u,x).%
\end{equation*}
Here we use that the periodic orbit $\varphi(\R,x,u)$ is compact and hence, by Step 1, $\alpha_{r(t)}(u,x)$ is bounded by a constant independent of $t$. Similarly, one shows%
\begin{equation*}
  \frac{1}{2\tau}\liminf_{\N\ni n\rightarrow\infty}\frac{1}{n}\alpha_{2\tau n}(u,x) \leq \liminf_{t\rightarrow\infty}\frac{1}{t}\alpha_t(u,x).%
\end{equation*}
We can write%
\begin{equation*}
  \left\|(\rmd\varphi_{2\tau n,u})_x^{\wedge}\right\| = \max_{1\leq j\leq d}\left\|(\rme^{2\tau nR})^{\wedge j}\right\| = \max_{1\leq j\leq d}\left\|(\rme^{2\tau R_j})^n\right\|,%
\end{equation*}
where $R_j$ is the $j$-th derivation operator induced by $R$. Writing $\RT(\lambda_1) \geq \cdots \geq \RT(\lambda_d)$ for the real parts of the eigenvalues of $R$ and $\rho(\cdot)$ for the spectral radius of an operator, we find%
\begin{eqnarray*}
  \frac{1}{n}\alpha_{2\tau n}(u,x) &=& \log^+\left(\max_{1\leq j\leq d}\left\|(\rme^{2\tau R_j})^n\right\|^{1/n}\right)\allowdisplaybreaks\\
	                                 &\xrightarrow{n\rightarrow\infty}& \max\left\{0,\max_{1\leq j\leq d} \log \rho\left(\rme^{2\tau R_j}\right)\right\}\allowdisplaybreaks\\
							     								 &=& \max\left\{0,\max_{1\leq j\leq d} \log \prod_{i=1}^j \rme^{2\tau\RT(\lambda_i)}\right\}\allowdisplaybreaks\\
										  		   			 &=& \max\left\{0,\log\prod_{i:\ \RT(\lambda_i)>0}\rme^{2\tau\RT(\lambda_i)}\right\}\allowdisplaybreaks\\
											   				   &=& 2\tau \sum_{i:\ \RT(\lambda_i)>0}\RT(\lambda_i).%
\end{eqnarray*}
Since the real parts of the $\lambda_i$ are the Lyapunov exponents, it follows that%
\begin{equation*}
  \limsup_{t\rightarrow\infty}\frac{1}{t}\alpha_t(u,x) \leq \lambda_1(u,x)+\cdots+\lambda_k(u,x) \leq \liminf_{t\rightarrow\infty}\frac{1}{t}\alpha_t(u,x)%
\end{equation*}
and the proof is complete.%
\end{proof}

In the rest of this subsection, we assume that $D$ is a control set with nonempty interior and compact closure such that%
\begin{equation}\label{eq_regass}
  F \mbox{ is of class } \CC^{\infty} \mbox{ and } a(x,u) = T_xM \mbox{ for all } (x,u) \in \inner D \tm \inner U.%  
\end{equation}

\begin{lemma}\label{lem_exregpertraj}
For every $x\in\inner D$ there exists a regular periodic controlled trajectory $(\varphi(\cdot,x,u),u(\cdot))$ with $u\in\inner\UC$.%
\end{lemma}

\begin{proof}
From Corollary \ref{cor_regtrajex} the existence of a smooth control function $u_1\in\CC^{\infty}(\R,\inner U)$ follows such that $(\varphi(\cdot,x,u_1),u_1(\cdot))$ is regular on an interval of the form $[0,\tau_1]$, where $\tau_1$ is chosen small enough that $\varphi([0,\tau_1],x,u_1) \subset \inner D$. Since \eqref{eq_regass} implies local accessibility on $\inner D$ and approximate controllability holds on $D$, we find a piecewise constant control function $u_2$ with values in $\inner U$ such that $\varphi(\tau_2,\varphi(\tau_1,x,u_1),u_2) = x$. The appropriate concatenation of $u_1$ and $u_2$ yields a periodic control function $u$ of period $\tau := \tau_1 + \tau_2$ such that the corresponding trajectory $\varphi(\cdot,x,u)$ is $\tau$-periodic as well. From Proposition \ref{prop_regularity} it follows that $(\varphi(\cdot,x,u),u(\cdot))$ is regular on $[0,\tau]$ and Lemma \ref{lem_intuchar} yields $u\in\inner\UC$.%
\end{proof}

The following approximation result allows us to get rid of the regularity assumption in Theorem \ref{thm_upperboundcs}.%

\begin{proposition}\label{prop_firstapprox}
Let $(\varphi(\cdot,x,u),u(\cdot))$ be a $\tau$-periodic controlled trajectory with $(u,x) \in \inner\UC \tm \inner D$. Moreover, let $\beta:\R\tm(\UC\tm M) \rightarrow \R_+$ be a nonnegative subadditive cocycle over the control flow such that for all $T>0$, $y\in M$ and $u_1,u_2\in\UC$ it holds that%
\begin{equation}\label{eq_cocycleass}
  u_1(t) = u_2(t) \mbox{ a.e.~on } [0,T] \quad\Rightarrow\quad \beta_T(u_1,y) = \beta_T(u_2,y).%
\end{equation}
Then for every $\ep>0$ there exists a regular periodic controlled trajectory $(\varphi(\cdot,x,u_*),u_*(\cdot))$ with $u_*\in\inner\UC$ and period $\tau_*>0$ so that%
\begin{equation*}
  \frac{1}{\tau_*}\beta_{\tau_*}(u_*,x) \leq \frac{1}{\tau}\beta_{\tau}(u,x) + \ep.%
\end{equation*}
\end{proposition} 

\begin{proof}
For the given periodic trajectory we construct a family of approximating trajectories. By Lemma \ref{lem_exregpertraj} there exists a regular periodic controlled trajectory $(\varphi(\cdot,x,v),v(\cdot))$ with $v\in\inner\UC$ and period $\rho>0$. For every $N\in\N$ we define%
\begin{equation*}
  u_N(t) := \left\{\begin{array}{rl}
	                      u(t) & \mbox{for } t\in[0,N\tau)\\
												v(t-N\tau) & \mbox{for } t\in[N\tau,N\tau+\rho]%
									 \end{array}\right.,%
\end{equation*}
and we extend $u_N$ $(N\tau+\rho)$-periodically. By construction and Lemma \ref{lem_intuchar}, $u_N$ is an admissible control function in $\inner\UC$. Moreover, by Proposition \ref{prop_regularity}, $u_N$ is regular on $[0,N\tau+\rho]$. Using subadditivity of $\beta$ and \eqref{eq_cocycleass}, we find%
\begin{equation*}
  \beta_{N\tau+\rho}(u_N,x) \leq N\beta_{\tau}(u,x) + \beta_{\rho}(v,x).%
\end{equation*}
Hence, for given $\ep>0$ we can choose $N$ sufficiently large so that%
\begin{eqnarray*}
   \frac{1}{N\tau+\rho}\beta_{N\tau+\rho}(u_N,x) &\leq& \frac{N}{N\tau+\rho}\beta_{\tau}(u,x) + \frac{1}{N\tau+\rho}\beta_{\rho}(v,x)\\
	                                           &\leq& \frac{1}{\tau + \rho/N}\beta_{\tau}(u,x) + \ep \leq \frac{1}{\tau}\beta_{\tau}(u,x) + \ep.%
\end{eqnarray*}
The assertion follows with $u_* = u_N$ and $\tau_* = N\tau+\rho$.%
\end{proof}

Applying this approximation result to the subadditive cocycle defined in \eqref{eq_defcocycle} yields the following estimate.%

\begin{proposition}\label{prop_perest}
For every compact set $K\subset D$ we have%
\begin{equation*}
  h_{\inv}(K,D) \leq \inf_{(u,x)}\lim_{t\rightarrow\infty}\frac{1}{t}\log^+\left\|(\rmd\varphi_{t,u})_x^{\wedge}\right\|,% 
\end{equation*}
where the infimum runs over all $\phi$-periodic points $(u,x) \in \inner\UC\tm\inner D$.%
\end{proposition}

\begin{proof}
Let $(\varphi(\cdot,x,u),u(\cdot))$ be a $\tau$-periodic controlled trajectory with $(u,x)\in\inner\UC\tm\inner D$ and consider the subadditive cocycle $\alpha_t(u,x) = \log^+\|(\rmd\varphi_{t,u})_x^{\wedge}\|$. By Proposition \ref{prop_perposlyapsum} we know that%
\begin{equation*}
  \sigma := \lim_{t\rightarrow\infty}\frac{1}{t}\alpha_t(u,x)%
\end{equation*}
exists. Hence, we can choose $n_0\in\N$ large enough so that%
\begin{equation}\label{eq_alimit}
  \left|\sigma - \frac{1}{n_0\tau}\alpha_{n_0\tau}(u,x)\right| \leq \frac{\ep}{2}.%
\end{equation}
Since $\alpha$ is nonnegative and satisfies assumption \eqref{eq_cocycleass} of Proposition \ref{prop_firstapprox}, there exists a regular $\tau_*$-periodic trajectory $(\varphi(\cdot,x,u_*),u_*(\cdot))$ with%
\begin{equation}\label{eq_aapprox}
  \frac{1}{\tau_*}\alpha_{\tau_*}(u_*,x) \leq \frac{1}{n_0\tau}\alpha_{n_0\tau}(u,x) + \frac{\ep}{2}.% 
\end{equation}
The sequence $n \mapsto \alpha_{n\tau_*}(u_*,x)$ is easily seen to be subadditive, which (using Fekete's subadditivity lemma) implies%
\begin{eqnarray*}
  \lim_{n\rightarrow\infty}\frac{1}{n\tau_*}\alpha_{n\tau_*}(u_*,x) &=& \inf_{n\in\N}\frac{1}{n\tau_*}\alpha_{n\tau_*}(u_*,x)\\
	                           &\leq& \frac{1}{\tau_*}\alpha_{\tau_*}(u_*,x) \stackrel{\eqref{eq_aapprox}}{\leq} \frac{1}{n_0\tau}\alpha_{n_0\tau}(u,x) + \frac{\ep}{2}.%
\end{eqnarray*}
Now Theorem \ref{thm_upperboundcs} together with Proposition \ref{prop_perposlyapsum} gives%
\begin{eqnarray*}
   h_{\inv}(K,D) &\leq& \lim_{t\rightarrow\infty}\frac{1}{t}\alpha_t(u_*,x) = \lim_{n\rightarrow\infty}\frac{1}{n\tau_*}\alpha_{n\tau_*}(u_*,x)\\
                 &\leq& \frac{1}{n_0\tau}\alpha_{n_0\tau}(u,x) + \frac{\ep}{2} \stackrel{\eqref{eq_alimit}}{\leq} \sigma + \ep.%
\end{eqnarray*}
Since $\ep$ can be chosen arbitrarily small, this completes the proof.%
\end{proof}

The next approximation result enables us to drop the periodicity assumption in Theorem \ref{thm_upperboundcs}.%

\begin{proposition}\label{prop_secondapprox}
Let $\beta:\R\tm(\UC\tm M) \rightarrow \R_+$ be a nonnegative subadditive cocycle over the control flow satisfying the following assumptions.%
\begin{enumerate}
\item[(a)] For every compact set $K\subset M$ there is a constant $C = C(K)\geq0$ so that $\varphi([0,\tau],x,u)\subset K$ implies $\beta_t(u,x) \leq Ct$, $t\in[0,\tau]$.%
\item[(b)] For all $T>0$, $y\in M$ and $u_1,u_2\in\UC$ it holds that%
\begin{equation*}
  u_1(t) = u_2(t) \mbox{ a.e.~on } [0,T] \quad\Rightarrow\quad \beta_T(u_1,y) = \beta_T(u_2,y).%
\end{equation*}
\end{enumerate}
Let $(u,x) \in \inner\UC \tm \inner D$ such that $\varphi(t,x,u)$ is contained in a compact set $K\subset \inner D$ for all $t\geq0$. Then for every $\ep>0$ there exists a $\tau_*$-periodic controlled trajectory $(\varphi(\cdot,x,u_*),u_*(\cdot))$, $u_*\in\inner\UC$, for some $\tau_*>0$ with%
\begin{equation*}
  \frac{1}{\tau_*}\beta_{\tau_*}(u_*,x) \leq \limsup_{t\rightarrow\infty}\frac{1}{t}\beta_t(u,x) + \ep.%
\end{equation*}
\end{proposition}

\begin{proof}
Let $(t_n)_{n\in\N}$ be a sequence of positive times with $t_n\rightarrow\infty$ such that%
\begin{equation*}
  \sigma := \limsup_{t\rightarrow\infty}\frac{1}{t}\beta_t(u,x) = \lim_{n\rightarrow\infty}\frac{1}{t_n}\beta_{t_n}(u,x).%
\end{equation*}
Define the first hitting time%
\begin{equation*}
  \tau := \inf\left\{t\geq0:\ x\in\OC^+_{\leq t}(z) \mbox{ for all } z \in K\right\}.%
\end{equation*}
By a general fact we have $\tau<\infty$ (cf.~\cite[Lem.~3.2.21]{CKl} or \cite[Prop.~1.23]{Ka2}). There is $n_1\in\N$ such that for all $n\geq n_1$ and $T\in[0,\tau]$,%
\begin{equation}\label{eq_522}
  \frac{1}{t_n+T}\sup_{(t,z,v)\in[0,\tau]\tm K\tm\UC\atop \varphi([0,\tau],z,v)\subset\cl D}\beta_t(v,z) \leq \frac{C\tau}{t_n+T} \leq \frac{\ep}{2},%
\end{equation}
where $C = C(\cl D)$. Finally, there is $N\geq n_1$ with%
\begin{equation}\label{eq_524}
  \left|\frac{1}{t_N}\beta_{t_N}(u,x) - \sigma\right| \leq \frac{\ep}{2}.%
\end{equation}
By definition of $\tau$ we can choose a control $v:[0,T] \rightarrow U$ with $T\leq\tau$ and $\varphi(T,\varphi(t_N,x,u),v) = x$, and we may assume that $v$ is piecewise constant taking values in $\inner U$. Let $\tau_* := t_N + T$, define $u_*$ on $[0,\tau_*]$ as%
\begin{equation*}
  u_*(t) := \left\{\begin{array}{rl}
	                    u(t) & \mbox{for } t\in[0,t_N],\\
										v(t-t_N) & \mbox{for } t\in(t_N,\tau_*]
									 \end{array}\right.,%
\end{equation*}
and extend $u_*$ $\tau_*$-periodically. Then $(\varphi(\cdot,x,u_*),u_*(\cdot))$ is a $\tau_*$-periodic controlled trajectory with $(u_*,x)\in\inner\UC\tm\inner D$. We obtain%
\begin{eqnarray*}
  \frac{1}{\tau_*}\beta_{\tau_*}(u_*,x) &\leq& \frac{1}{t_N+T}(\beta_{t_N}(u_*,x) + \beta_T(\Theta_{t_N}u_*,\varphi(t_N,x,u_*)))\\
	 &=& \frac{1}{t_N+T}(\beta_{t_N}(u,x) + \beta_T(v,\varphi(t_N,x,u)))\\
	&\stackrel{\eqref{eq_522}}{\leq}& \frac{1}{t_N+T}\beta_{t_N}(u,x) + \frac{\ep}{2}\\
	&\leq& \frac{1}{t_N}\beta_{t_N}(u,x) + \frac{\ep}{2} \stackrel{\eqref{eq_524}}{\leq} \sigma + \ep,%
\end{eqnarray*}
which completes the proof.%
\end{proof}

The subadditive cocycle defined in \eqref{eq_defcocycle} satisfies the assumptions of the above proposition, which yields the main result of this section.%

\begin{theorem}\label{thm_genupperbound}
Assume that the control system $\Sigma$ satisfies the regularity assumption \eqref{eq_regass} on a control set $D$ with nonempty interior and compact closure. Then for every compact set $K\subset D$ it holds that%
\begin{equation*}
   h_{\inv}(K,D) \leq \inf_{(u,x)}\limsup_{t\rightarrow\infty}\frac{1}{t}\log^+\left\|(\rmd\varphi_{t,u})_x^{\wedge}\right\|,%
\end{equation*}
where the infimum runs over all $(u,x) \in \inner\UC\tm\inner D$ such that $\varphi(t,x,u)$ is contained in a compact subset of $\inner D$ for all $t\geq0$.%
\end{theorem}

\begin{proof}
We apply Proposition \ref{prop_secondapprox} to the subadditive cocycle $\alpha_t(u,x) = \log^+\|(\rmd\varphi_{t,u})_x^{\wedge}\|$. Given a trajectory $\varphi(\cdot,x,u)$ with values in a compact subset of $\inner D$ and $\ep>0$, we find a $\tau_*$-periodic controlled trajectory $(\varphi(\cdot,x,u_*),u_*(\cdot))$ with $u_*\in\inner\UC$ such that%
\begin{equation*}
  \frac{1}{\tau_*}\alpha_{\tau_*}(u_*,x) \leq \limsup_{t\rightarrow\infty}\frac{1}{t}\alpha_t(u,x) + \ep.%
\end{equation*}
Then Proposition \ref{prop_perest} yields%
\begin{eqnarray*}
  h_{\inv}(K,D) &\leq& \lim_{t\rightarrow\infty}\frac{1}{t}\alpha_t(u_*,x) = \lim_{\N\ni n\rightarrow\infty}\frac{1}{n\tau_*}\alpha_{n\tau_*}(u_*,x)\\
									 &=& \inf_{n\in\N}\frac{1}{n\tau_*}\alpha_{n\tau_*}(u_*,x) \leq \frac{1}{\tau_*}\alpha_{\tau_*}(u_*,x) \leq \limsup_{t\rightarrow\infty}\frac{1}{t}\alpha_t(u,x) + \ep,%
\end{eqnarray*}
implying the assertion.%
\end{proof}

The following corollary shows that for control-affine systems the strong jet accessibility assumption \eqref{eq_regass} can be weakened to local accessibility.%

\begin{corollary}\label{cor_regforcasys}
Assume that the system $\Sigma$ is control-affine, i.e., $F(x,u) = f_0(x) + \sum_{i=1}^m u_i f_i(x)$ with $\CC^{\infty}$-vector fields $f_0,f_1,\ldots,f_m$ and a compact and convex control range $U$ with $0\in\inner U$. Then Theorem \ref{thm_genupperbound} also holds under the assumption that the system satisfies the classical accessibility rank condition on $\inner D$, i.e., if the Lie algebra generated by $f_0,f_1,\ldots,f_m$ has full rank at every $x\in\inner D$ (instead of $a(x,u)=T_xM$ on $\inner D \tm \inner U$).%
\end{corollary}

\begin{proof}
The proof is subdivided into two steps.%

\emph{Step 1.} For each $\gamma>1$ we consider the time-transformed system%
\begin{equation*}
  \Sigma^{\gamma}:\quad \dot{x}(t) = v(t) \cdot F(x(t),u(t)),\quad (v,u) \in \UC^{\gamma} = \VC^{\gamma} \tm \UC,%
\end{equation*}
where $\VC^{\gamma} := \{v\in L^{\infty}(\R,\R) : v(t) \in [1/\gamma,\gamma] \mbox{ a.e.}\}$. The trajectories of $\Sigma^{\gamma}$ are just time reparametrizations of the trajectories of $\Sigma$. To show this, for every $v\in\VC^{\gamma}$ define%
\begin{equation*}
  \sigma_v(t) := \int_0^t v(s)\rmd s,\quad t \in \R.%
\end{equation*}
It is clear that $\sigma_v$ is locally absolutely continuous with $\sigma(0)=0$. The inequality $v(s)\geq 1/\gamma>0$ implies that $\sigma_v$ is strictly increasing with $\sigma_v(t) \rightarrow \pm \infty$ for $t\rightarrow\pm\infty$. We claim that%
\begin{equation}\label{eq_reparam}
  \varphi(\sigma_v(t),x,u) \equiv \varphi^{\gamma}(t,x,(v,u \circ \sigma_v)),%
\end{equation}
where $\varphi^{\gamma}$ is the transition map associated with $\Sigma^{\gamma}$. Indeed, for almost all $t\in\R$,%
\begin{eqnarray*}
  \frac{\rmd}{\rmd t}\varphi(\sigma_v(t),x,u) &=& \dot{\sigma}_v(t) \cdot F(\varphi(\sigma_v(t),x,u),u(\sigma_v(t)))\\
	                                            &=& v(t) \cdot F(\varphi(\sigma_v(t),x,u),(u \circ \sigma_v)(t)).%
\end{eqnarray*}
Then uniqueness of solutions yields \eqref{eq_reparam}. From this identity it can easily be seen that a control set $D$ of $\Sigma$ is also a control set of $\Sigma^{\gamma}$. We claim that the invariance entropies of an admissible pair $(K,D)$ w.r.t.~$\Sigma$ and $\Sigma^{\gamma}$ satisfy%
\begin{equation}\label{eq_entropiesineq}
   h_{\inv}(K,D;\Sigma) \leq \gamma \cdot h_{\inv}(K,D;\Sigma^{\gamma}).%
\end{equation}
To prove this, let $\SC \subset \VC^{\gamma} \tm \UC$ be a $(\tau,K,D)$-spanning set for $\Sigma^{\gamma}$. Then%
\begin{equation*}
  \SC' := \left\{ u \circ \sigma_v^{-1} \ : \ (v,u) \in \SC \right\}%
\end{equation*}
is a $(\tau/\gamma,K,D)$-spanning set for $\Sigma$. Indeed, for $x\in K$ there exists $(v,u) \in \SC$ with $\varphi^{\gamma}(t,x,(v,u)) = \varphi(\sigma_v(t),x,u \circ \sigma_v^{-1}) \in D$ for $t\in[0,\tau]$. The claim follows, because $\sigma_v(\tau) = \int_0^{\tau}v(s)\rmd s \geq \tau/\gamma$. Hence, $r_{\inv}(\tau/\gamma,K,D;\Sigma) \leq r_{\inv}(\tau,K,D;\Sigma^{\gamma})$ implying \eqref{eq_entropiesineq}.%

\emph{Step 2.} We show that the time-transformed systems $\Sigma^{\gamma}$ satisfy the strong accessibility rank condition if $\Sigma$ satisfies the classical accessibility rank condition. The strong accessibility algebra of $\Sigma^{\gamma}$ contains the differences%
\begin{equation*}
  v\left(f_0 + \sum_{i=1}^m u_i f_i\right) - v'\left(f_0 + \sum_{i=1}^m u_i' f_i\right),\quad u,u' \in U,\ v,v' \in \left[\gamma^{-1},\gamma\right],%
\end{equation*}
and hence it contains the vector fields $f_1,\ldots,f_m$ as well as $f_0$ (put $(v,u) := (\gamma,0) \in \R\tm\R^m$ and $(v',u') := (1,0) \in \R\tm\R^m$). Consequently, since $\LC(f_0,f_1,\ldots,f_m)$ has full rank at every point in $\inner D$ by assumption, the same holds for the strong accessibility algebra of $\Sigma^{\gamma}$. Since the systems $\Sigma^{\gamma}$ are polynomial in $u$, we have equalities in \eqref{eq_14}. Hence, Theorem \ref{thm_genupperbound} can be applied to $\Sigma^{\gamma}$ for every $\gamma>1$. This gives%
\begin{eqnarray*}
  h_{\inv}(K,D;\Sigma^{\gamma}) &\leq& \inf_{((v,u),x)}\limsup_{t\rightarrow\infty}\frac{1}{t}\log^+\left\|(\rmd\varphi^{\gamma}_{t,u})_x^{\wedge}\right\|\\
	                                 &=& \inf_{((v,u),x)}\limsup_{t\rightarrow\infty}\frac{1}{t}\log^+\left\|(\rmd\varphi_{\sigma_v(t),u \circ \sigma_v^{-1}})_x^{\wedge}\right\|\\
															  &\leq& \gamma \inf_{(u,x)}\limsup_{s\rightarrow\infty}\frac{1}{s}\log^+\left\|(\rmd\varphi_{s,u})_x^{\wedge}\right\|.%
\end{eqnarray*}
Letting $\gamma\rightarrow 1$ in \eqref{eq_entropiesineq} then implies the assertion.%
\end{proof}

\section{Lower Bounds for Hyperbolic Sets}\label{sec_glb}%
  
In this section, we derive a lower bound for the invariance entropy of a hyperbolic controlled invariant set. The proof of this result is mainly based on an estimate of the volumes of Bowen-balls, a skew-product version of the well-known Bowen-Ruelle volume lemma. In the following subsection, we provide a detailed proof of this lemma.%

\subsection{The Volume Lemma}%

First we prove a version of the Bowen-Ruelle volume lemma for skew products and apply it to control-affine systems. This lemma can be found in Liu \cite{Liu} in a special formulation for discrete-time random hyperbolic sets. Here we will give a detailed proof in the deterministic case. We prove the lemma for discrete-time skew-products and then transfer it to the continuous-time case via discretization.%

We consider a discrete-time skew-product%
\begin{equation*}
  \phi:\Z \tm B \tm M \rightarrow B \tm M,\quad (k,(b,x)) \mapsto \phi_k(b,x) = (\theta^kb,\varphi(k,x,b)),%
\end{equation*}
where $B$ is a compact metric space and $(M,g)$ a Riemannian manifold. Moreover, we assume the following:%
\begin{enumerate}
\item[(i)] The base map $\theta:B\rightarrow B$ and the cocycle $\varphi:\Z\tm B\tm M\rightarrow M$ are continuous.%
\item[(ii)] For every $k\in\Z$ and $b\in B$, the map $\varphi_{k,b}:M\rightarrow M$, $x\mapsto\varphi(k,x,b)$, is a $\CC^2$-diffeomorphism and both its first and its second derivative depends continuously on $(b,x)$.%
\end{enumerate}
Let $Q\subset M$ be a compact set such that for every $x\in Q$ there exists $b\in B$ with $\varphi(\Z,x,b) \subset Q$. Then we define the lift of $Q$ to $B\tm M$ by%
\begin{equation*}
  \QC := \left\{(b,x)\in B\tm M\ :\ \varphi(\Z,x,b) \subset Q\right\}.%
\end{equation*}
It is easy to show that $\QC$ is a compact $\phi$-invariant set. We further assume that for every $(b,x)\in\QC$ there exists a splitting%
\begin{equation*}
  T_xM = E^-_{b,x} \oplus E^+_{b,x}%
\end{equation*}
with the following properties:%
\begin{enumerate}
\item[(a)] The subspaces $E^{\pm}_{b,x}$ are invariant, i.e., for all $(b,x)\in \QC$ and $k\in\Z$,%
\begin{equation*}
  (\rmd\varphi_{k,b})_xE^{\pm}_{b,x} = E^{\pm}_{\phi_k(b,x)}.%
\end{equation*}
\item[(b)] There are constants $c\geq1$ and $\lambda\in(0,1)$ such that%
\begin{equation*}
  \left|(\rmd\varphi_{k,b})_xv\right| \leq c\lambda^k|v|,\quad\forall k\geq0,\ (b,x)\in\QC,\ v\in E^-_{b,x},%
\end{equation*}
and%
\begin{equation*}
  \left|(\rmd\varphi_{k,b})_xv\right| \geq c^{-1}\lambda^{-k}|v|,\quad\forall k\geq0,\ (b,x)\in\QC,\ v\in E^+_{b,x}.%
\end{equation*}
\item[(c)] The subspaces $E^{\pm}_{b,x}$ vary continuously with $(b,x)$, i.e., the projections $\pi^{\pm}_{b,x}:T_xM \rightarrow E^{\pm}_{b,x}$ along $E^{\mp}_{b,x}$ depend continuously on $(b,x)$.%
\end{enumerate}
For every $b\in B$ and $n\in\N_0$ we introduce the \emph{Bowen-metric}%
\begin{equation*}
  \varrho_{b,n}(x,y) := \max_{0\leq k\leq n}\varrho(\varphi(k,x,b),\varphi(k,y,b)),%
\end{equation*}
which is a metric on $M$ topologically equivalent to $\varrho$. The open balls%
\begin{equation*}
  B^n_b(x,\ep) = \left\{y\in M:\ \varrho_{b,n}(x,y)<\ep\right\}%
\end{equation*}
are called \emph{Bowen-balls (of order $n$ and radius $\ep$)}.%

\begin{lemma}\label{lem_poly}
The following polynomials map the interval $[0,1]$ onto itself.%
\begin{eqnarray*}
  p_1(x) &=& x^n + n(1-x)x^{n-1},\\
	p_2(x) &=& x^i + \left[\frac{n}{2}\right](1-x^2)(1-x)x^{i-1} + (1-x)x^{n-i+1},%
\end{eqnarray*}
where $n\geq2$, and $[n/2] \leq i \leq n$.%
\end{lemma}

\begin{proof}
Obviously, we have $p_1(0)=0$ and $p_1(x)\geq0$ for all $x\in[0,1]$. For $x>0$ we obtain by Bernoulli's inequality that%
\begin{equation*}
  p_1(x) = x^n\left(1 + n\frac{1-x}{x}\right) \leq x^n \left(1 + \frac{1-x}{x}\right)^n = 1.%
\end{equation*}
Also $p_2$ satisfies $p_2(0)=0$ and $p_2(x)\geq0$ for $x\in[0,1]$. For $n\geq2$ and $x \in (0,1)$, using Bernoulli again, we can show that%
\begin{eqnarray*}
  p_2(x) &\leq& x^i + i(1-x^2)(1-x)x^{i-1} + (1-x)x\\
	          &=& \left(x^i + (1-x)x\right)\left(1 + i\frac{(1-x^2)(1-x)x^{i-1}}{x^i + (1-x)x}\right)\\
				 &\leq& \left(x^i + (1-x)x\right)\left(1 + \frac{(1-x^2)(1-x)x^{i-1}}{x^i + (1-x)x}\right)^i\\
						&=& \left(x^i + (1-x)x\right)^{1-i}\left((1-x^2+x^3)x^{i-1} + x(1-x)\right)^i.%
\end{eqnarray*}
Now it suffices to show that both factors are between $0$ and $1$. For the first one we have $x^i + (1-x)x \leq x + (1-x)x = x(2-x) \leq 1$ and for the second one%
\begin{equation*}
  (1-x^2+x^3)x^{i-1} + x(1-x) \leq x^{i-1} + x(1-x) \leq x(2-x) \leq 1.%
\end{equation*}
The proof is finished.%
\end{proof}

\begin{vlemma}\label{lem_vlem}
Assume that the dimensions of the subspaces $E^{\pm}_{b,x}$ are constant on $\QC$ and that the constant $c$ in (b) equals $1$. Then for every sufficiently small $\ep>0$ there exists $C_{\ep}\geq1$ such that for all $(b,x)\in\QC$ and $n\geq0$,%
\begin{equation*}
  C_{\ep}^{-1} \leq \vol\left(B^n_b(x,\ep)\right)\left|\det\left(\rmd\varphi_{n,b}|_{E^+_{b,x}}:E^+_{b,x} \rightarrow E^+_{\phi_n(b,x)}\right)\right| \leq C_{\ep}.%
\end{equation*}
\end{vlemma}

\begin{proof}
The proof is subdivided into nine steps.%

\emph{Step 1.} In the first step we prove three claims.%
\begin{enumerate}
\item[(i)] There exists a constant $K\geq 1$ such that for all $(b,x)\in\QC$ and $T_xM \ni v = v^+ \oplus v^- \in E^+_{b,x} \oplus E^-_{b,x}$ it holds that%
\begin{equation}\label{eq_step1_norms}
  \frac{1}{2}|v| \leq \|v\|_{b,x} := \max\{|v^-|,|v^+|\} \leq K|v|.%
\end{equation}
This is proved as follows. We have $v^{\pm} = \pi^{\pm}_{b,x}(v)$ and the projections depend continuously on $(b,x)$. By compactness of $\QC$, the suprema $\sup_{(b,x)\in\QC}\|\pi_{b,x}^{\pm}\|$ are finite. Then \eqref{eq_step1_norms} holds for all $K\geq1$ bigger than the maximum of these two suprema. The lower estimate follows from the triangle inequality $|v|\leq |v^+| + |v^-| \leq 2\max\{|v^+|,|v^-|\}$.%
\item[(ii)] There is $r_0>0$ such that  for any $(b,x)\in\QC$ the maps%
\begin{eqnarray*}
  \widetilde{\varphi}_{b,x} &:=& \exp^{-1}_{\varphi_{1,b}(x)} \circ \varphi_{1,b} \circ \exp_x:\{v\in T_xM : |v|\leq r_0\} \rightarrow T_{\varphi_{1,b}(x)}M,\\
	\widetilde{\varphi}^-_{b,x} &:=& \exp_x^{-1} \circ \varphi^{-1}_{1,b} \circ \exp_{\varphi_{1,b}(x)}:\{w\in T_{\varphi_{1,b}(x)}M : |w|\leq r_0\} \rightarrow T_xM%
\end{eqnarray*}
are well-defined. Since $Q$ is compact, we find $r_1>0$ such that $\exp_x$ is defined on the closed $r_1$-ball around $0_x\in T_xM$ for every $x\in Q$. Since the derivative of $\varphi_{1,b}$ depends continuously on $(b,x)$, there is a common Lipschitz constant $L$ for the maps $\varphi_{1,b}$, $b\in B$, on a small compact neighborhood of $Q$. Let $s>0$ be chosen such that $\exp_y^{-1}$ is defined on the closed ball $\cl B(y,s)$ for every $y\in Q$ and let $r_2 := s/L$. With $r_0 := \min\{r_1,r_2\}$ we find that $\widetilde{\varphi}_{b,x}$ is well-defined. The same arguments apply to $\widetilde{\varphi}^-_{b,x}$.%
\item[(iii)] There is $A_0\geq1$ such that for all $(b,x)\in\QC$ it holds that%
\begin{equation}\label{eq_step1_opnorms}
  \sup_{v\in T_xM,|v|\leq r_0}\|(\rmd\widetilde{\varphi}_{b,x})_v\|,\quad \sup_{w\in T_{\varphi_{1,b}(x)}M,|w|\leq r_0}\|(\rmd\widetilde{\varphi}^-_{b,x})_w\| \leq A_0,%
\end{equation}
and%
\begin{equation}\label{eq_step1_lipconstants}
  \Lip\left((\rmd\widetilde{\varphi}_{b,x})_{(\cdot)}\right),\quad \Lip\left((\rmd\widetilde{\varphi}^-_{b,x})_{(\cdot)}\right) \leq A_0,%
\end{equation}
where the Lipschitz constants are taken w.r.t.~$|\cdot|$. To show \eqref{eq_step1_opnorms}, we use the chain rule, which, e.g., gives%
\begin{equation*}
  \left\|(\rmd\widetilde{\varphi}_{b,x})_v\right\| \leq \left\|(\rmd\exp^{-1}_{\varphi_{1,b}(x)})_{\varphi_{1,b}(\exp_x(v))}\right\| \cdot \left\|(\rmd\varphi_{1,b})_{\exp_x(v)}\right\|\cdot\left\|(\rmd\exp_x)_v\right\|.%
\end{equation*}
Since the right-hand side is continuous in $(b,x,v)$ and the subset of $TM$ consisting of all vectors $v$ with $|v|\leq r_0$ and $v\in T_xM$ for some $x\in Q$ is compact, we find that there exists a bound as claimed (analogously for $\widetilde{\varphi}^-_{b,x}$). For the proof of \eqref{eq_step1_lipconstants} we use the assumption that the second derivative of $\varphi_{1,b}$ depends continuously on $(b,x)$. This implies that the same is true for the maps $\widetilde{\varphi}_{b,x}$ and hence the first derivative is Lipschitz continuous and the Lipschitz constants (the norms of the second derivatives) are bounded on the compact set considered here (analogously for $\widetilde{\varphi}^-_{b,x}$).%
\end{enumerate}

\emph{Step 2.} In the second step we fix some constants. Let $u$ be the common dimension of the subspaces $E^+_{b,x}$ and put $A_1 := 4u(2A_0)^u$. Fix an $\ep_0>0$ with%
\begin{equation}\label{eq_step2_ep0}
  \ep_0 < \min\left\{\frac{1}{2}(1-\lambda),\frac{1}{2}(\lambda^{-1}-1)\right\} \mbox{\quad and\quad} (\lambda + \ep_0)^{-1} - \ep_0 > 1,%
\end{equation}
where the latter is equivalent to $1-\lambda > \ep_0(1+\lambda+\ep_0)$. Define%
\begin{equation*}
  \mu := (\lambda + 2\ep_0)\left[(\lambda+\ep_0)^{-1} - \ep_0\right]^{-1} < 1,\qquad \alpha := \frac{1 + \ep_0/2}{\lambda^{-1} - \ep_0/2} < 1.%
\end{equation*}
Choose $\beta \in (0,1)$ such that%
\begin{equation*}
  \beta \geq \max\{\alpha,\mu\} \mbox{\quad and \quad} 1-\beta^2 \leq \beta.%
\end{equation*}
Finally, choose $r>0$ such that%
\begin{equation}\label{eq_step2_rcond}
  r \leq \min\left\{r_0,\frac{\ep_0}{4KA_0},\frac{(1-\beta)(1-\beta^2)}{2K^2A_0A_1}\right\},%
\end{equation}
and for any $(b,x)\in\QC$%
\begin{equation}\label{eq_step2_lipconstants}
  \Lip(R_{b,x}),\quad \Lip(R^-_{b,x}) \leq \frac{\ep_0}{2K},%
\end{equation}
where $R_{b,x}$ is defined as the restriction of $\widetilde{\varphi}_{b,x} - (\rmd\widetilde{\varphi}_{b,x})_{0_x}$ to the set $\{v\in T_xM : |v|\leq r\}$, $R^-_{b,x}$ is defined analogously, and the Lipschitz constants are taken w.r.t.~the norm $|\cdot|$. Define for $(b,x)\in\QC$ and $n\geq0$%
\begin{equation*}
  D_b(x,n,r) := \left\{v\in T_xM\ :\ \left|\widetilde{\varphi}^k_{b,x}(v)\right| \leq r,\ k = 0,1,\ldots,n\right\},%
\end{equation*}
where $\widetilde{\varphi}^k_{b,x} := \widetilde{\varphi}_{\phi_{k-1}(b,x)} \circ \cdots \circ \widetilde{\varphi}_{\phi_1(b,x)} \circ \widetilde{\varphi}_{b,x}$.%

\emph{Step 3.} We claim that for all $(b,x)\in\QC$, $n\geq0$ and $v\in D_b(x,n,r)$ it holds that%
\begin{equation}\label{eq_step3_claim}
  \left|\widetilde{\varphi}^k_{b,x}(v)\right| \leq 2Kr\max\left\{\alpha^k,\alpha^{n-k}\right\}\quad \mbox{for } k = 0,1,\ldots,n.%
\end{equation}
Writing the derivative of $\widetilde{\varphi}_{b,x}$ at $0_x\in T_xM$ in the form%
\begin{equation*}
  (\rmd\widetilde{\varphi}_{b,x})_{0_x} = \left(\begin{array}{cc}
	                                                 A_{b,x} & C_{b,x} \\ D_{b,x} & B_{b,x} \end{array}\right):E^-_{b,x} \oplus E^+_{b,x} \rightarrow E^-_{\phi_1(b,x)} \oplus E^+_{\phi_1(b,x)},%
\end{equation*}
by invariance of the subbundles $E^{\pm}$, we can express $\widetilde{\varphi}_{b,x}$ as%
\begin{equation}\label{eq_step3_firstorderapprox}
  \widetilde{\varphi}_{b,x} = \left(\begin{array}{cc}
																					A_{b,x} & 0 \\ 0 & B_{b,x} \end{array}\right) + R_{b,x}(\cdot),%
\end{equation}
where the linear maps $A_{b,x}:E^-_{b,x} \rightarrow E^-_{\phi_1(b,x)}$ and $B_{b,x}:E^+_{b,x} \rightarrow E^+_{\phi_1(b,x)}$ satisfy $\|A_{b,x}\|\leq\lambda$ and $\|B^{-1}_{b,x}\|\leq\lambda$ with $R_{b,x}$ as defined in Step 2. From the definition of $K$ and \eqref{eq_step2_lipconstants} we conclude%
\begin{eqnarray*}
  \left|\widetilde{\varphi}_{b,x}(v)^+\right| &=& \left|B_{b,x}v^+ + R_{b,x}(v)^+\right| \geq \left|B_{b,x}v^+\right| - \left|R_{b,x}(v)^+\right|\\
	&\geq& \lambda^{-1}|v^+| - \frac{\ep_0}{2}|v| \geq \left(\lambda^{-1} - \frac{\ep_0}{2}\right)|v^+| - \frac{\ep_0}{2}|v^-|,%
\end{eqnarray*}
which, putting $\alpha' := (\lambda^{-1} - \ep_0/2)^{-1}$, implies%
\begin{equation}\label{eq_step3_vplusineq}
  |v^+| \leq \alpha'\left(|\widetilde{\varphi}_{b,x}(v)^+| + \frac{\ep_0}{2}|v^-|\right) \mbox{\quad for } |v|\leq r.%
\end{equation}
Analogously, one shows%
\begin{equation}\label{eq_step3_vminusineq}
  |v^-| \leq \alpha'\left(|\widetilde{\varphi}^-_{b,x}(v)^-| + \frac{\ep_0}{2}|v^+|\right) \mbox{\quad for } |v|\leq r.%
\end{equation}
Estimate \eqref{eq_step3_claim} now follows from an iterated application of \eqref{eq_step3_vplusineq} and \eqref{eq_step3_vminusineq}, where the number of iterations is $\min\{k,n-k\}$. For instance, let $n=5$, $k=2$. Then%
\begin{eqnarray*}
  \left|\widetilde{\varphi}^2_{b,x}(v)^+\right| &\leq& \alpha'\left(\left|\widetilde{\varphi}^3_{b,x}(v)^+\right| + \frac{\ep_0}{2}\left|\widetilde{\varphi}^2_{b,x}(v)^-\right|\right)\\
	&\leq& \alpha'\Bigl(\alpha'\left(\left|\widetilde{\varphi}^4_{b,x}(v)^+\right| + \frac{\ep_0}{2}\left|\widetilde{\varphi}^3_{b,x}(v)^-\right|\right)\\
	  && + \frac{\ep_0}{2}\alpha'\left(\left|\widetilde{\varphi}_{b,x}(v)^-\right| + \frac{\ep_0}{2}\left|\widetilde{\varphi}^2_{b,x}(v)^+\right|\right)\Bigr),%
\end{eqnarray*}
and analogously%
\begin{eqnarray*}
  \left|\widetilde{\varphi}^2_{b,x}(v)^-\right| &\leq& \alpha'\left(\left|\widetilde{\varphi}_{b,x}(v)^-\right| + \frac{\ep_0}{2}\left|\widetilde{\varphi}^2_{b,x}(v)^+\right|\right)\\
	&\leq& \alpha'\Bigl(\alpha'\left(\left|v^-\right| + \frac{\ep_0}{2}\left|\widetilde{\varphi}_{b,x}(v)^+\right|\right)\\
	  && + \frac{\ep_0}{2}\alpha'\left(\left|\widetilde{\varphi}_{b,x}^3(v)^+\right| + \frac{\ep_0}{2}\left|\widetilde{\varphi}^2_{b,x}(v)^-\right|\right)\Bigr).%
\end{eqnarray*}
Since we assume $v \in D_b(x,n,r)$, all the norms in these inequalities are bounded by $Kr$, implying%
\begin{equation*}
  \left|\widetilde{\varphi}^2_{b,x}(v)\right| \leq 2 Kr (\alpha')^2 \left(1 + \frac{\ep_0}{2}\right)^2 = 2Kr \alpha^2.%
\end{equation*}
For arbitrary $n$ and $k$ the proof works analogously.%

\emph{Step 4.} Let $(b,x)\in\QC$ and $v\in T_xM$ with $|v|\leq r$. Assume that $V$ is a subspace of $T_xM$ with $V \oplus E^-_{b,x} = T_xM$ and let $L_V:E^+_{b,x} \rightarrow E^-_{b,x}$ be the linear map such that $V = \Graph(L_V) = \{ v + L_Vv : v\in E^+_{b,x}\}$ (every $v^+ \in E^+_{b,x}$ can be written uniquely as $v^+ = v_1 \oplus v_2 \in E^-_{b,x} \oplus V$ and hence $L_V$ is given by $L_V(v^+) = -v_1$). Write $\theta_0(V) := \|L_V\|$ and $\theta_k(v,V) := \theta_0((\rmd\widetilde{\varphi}^k_{b,x})_vV)$, $k\geq1$, if well-defined. We claim that $\theta_0(V) \leq 1$ implies that $\theta_1(v,V)$ is well-defined with%
\begin{equation}\label{eq_step4_theta1est}
  \theta_1(v,V) \leq \mu\theta_0(V) + KA_0|v|.%
\end{equation}
To show that $\theta_1(v,V)$ is defined, we need to verify that $W:=(\rmd\widetilde{\varphi}_{b,x})_vV$ is a complement of $E^-_{\phi_1(b,x)}$. To this end, take $w\in V\backslash\{0\}$. Then%
\begin{eqnarray*}
  \left|(\rmd\widetilde{\varphi}_{b,x})_v(w)^+\right| &\geq& \left|(\rmd\widetilde{\varphi}_{b,x})_0(w)^+\right| - \left|(\rmd\widetilde{\varphi}_{b,x})_v(w)^+ - (\rmd\widetilde{\varphi}_{b,x})_0(w)^+\right|\\
                                                         &\geq& \lambda^{-1}|w^+| - KA_0|v||w|\\
                                                         &\stackrel{\eqref{eq_step2_rcond}}{\geq}& \lambda^{-1}|w^+| - KA_0 \frac{\ep_0}{4KA_0} |w^+ + L_Vw^+|\\
                                                         &\geq& \left(\lambda^{-1} - \frac{\ep_0}{2}\right)|w^+| \stackrel{\eqref{eq_step2_ep0}}{>} |w^+| > 0.%
\end{eqnarray*}
This gives $W \cap E^-_{\phi_1(b,x)} = \{0\}$ as desired. To prove \eqref{eq_step4_theta1est}, note that for every $w\in V$, by \eqref{eq_step3_firstorderapprox}, we can write%
\begin{equation*}
  (\rmd\widetilde{\varphi}_{b,x})_vw = \left(B_{b,x}w^+ + \left[(\rmd R_{b,x})_vw\right]^+\right) + \left(A_{b,x}w^- + \left[(\rmd R_{b,x})_vw\right]^-\right).%
\end{equation*}
The map $L_W$ satisfies%
\begin{equation*}
  L_W\left(B_{b,x}w^+ + \left[(\rmd R_{b,x})_vw\right]^+\right) = A_{b,x}L_V(w^+) + \left[(\rmd R_{b,x})_vw\right]^-.%
\end{equation*}
Hence, we can conclude%
\begin{equation*}
  \left|L_W\left(B_{b,x}w^+ + \left[(\rmd R_{b,x})_vw\right]^+\right)\right| \leq \lambda\theta_0(V)|w^+| + KA_0|v||w|,%
\end{equation*}
where we use $(\rmd R_{b,x})_v - (\rmd R_{b,x})_{0_x} = (\rmd\widetilde{\varphi}_{b,x})_v - (\rmd\widetilde{\varphi}_{b,x})_{0_x}$ and \eqref{eq_step1_lipconstants}. Furthermore,%
\begin{equation*}
  \left|B_{b,x}w^+ + \left[(\rmd R_{b,x})_vw\right]^+\right| \geq \lambda^{-1}|w^+| - KA_0|v||w|.%
\end{equation*}
Note that $|w| = |w^+ + L_Vw^+| \leq (1 + \theta_0(V))|w^+|$ implying%
\begin{equation*}
  \theta_1(v,V) = \|L_W\| \leq \frac{\lambda\theta_0(V) + KA_0|v|(1 + \theta_0(V))}{\lambda^{-1} - KA_0|v|(1 + \theta_0(V))}.%
\end{equation*}
Hence, to prove \eqref{eq_step4_theta1est} it suffices to show that%
\begin{equation*}
  \frac{\lambda\theta_0(V) + KA_0|v|(1 + \theta_0(V))}{\lambda^{-1} - KA_0|v|(1 + \theta_0(V))} \leq \mu\theta_0(V) + KA_0|v|.%
\end{equation*}
To this end, let $\kappa := KA_0|v|$ and $\rho := \theta_0(V)$. Then the above is equivalent to%
\begin{eqnarray*}
   \lambda\rho \!\!\!\!&+&\!\!\!\! \kappa(1+\rho) \leq (\mu\rho + \kappa)(\lambda^{-1} - \kappa(1+\rho)) \\
	  &\Leftrightarrow& \kappa(1+\rho)[1+\mu\rho+\kappa] \leq \lambda^{-1}(\mu\rho + \kappa) - \lambda\rho\\
		&\Leftrightarrow& (1+\rho)\kappa^2 + \left[(1+\rho)(1+\mu\rho) - \lambda^{-1}\right]\kappa + \rho\left(\lambda - \lambda^{-1}\mu\right) \leq 0.%
\end{eqnarray*}
A simple computation shows that for $\kappa=0$ the last inequality holds:%
\begin{equation*}
  \lambda - \lambda^{-1}\mu \leq 0 \quad\Leftrightarrow\quad 0 \leq \lambda^3\ep_0 + \lambda^2\ep_0^2 + \lambda\ep_0 + 2\lambda\ep_0 + 2\ep_0^2.%
\end{equation*}
By \eqref{eq_step2_rcond} we have $\kappa \leq KA_0r \leq \ep_0/4$. Since the left-hand side of the inequality is a quadratic polynomial in $\kappa$ whose highest-order coefficient $(1+\rho)$ is positive, it suffices to prove that the inequality holds for $\kappa = \ep_0/4$:%
\begin{equation}\label{eq_step4_kappamax}
  (1+\rho)\frac{\ep_0^2}{16} + \left[(1+\rho)(1+\mu\rho) - \lambda^{-1}\right]\frac{\ep_0}{4} + \rho\left(\lambda - \lambda^{-1}\mu\right) \leq 0.%
\end{equation}
For $\rho=0$ this is easily seen to be true. We check it for $\rho=1$:%
\begin{eqnarray*}
  \frac{\ep_0^2}{8} \!\!\!&+&\!\!\! \left[2(1+\mu) - \lambda^{-1}\right]\frac{\ep_0}{4} + \left(\lambda - \lambda^{-1}\mu\right) \leq 0\\
	&\Leftrightarrow& \frac{\ep_0^2}{8} + \frac{\ep_0}{2} - \frac{\ep_0}{4}\lambda^{-1} + \lambda  \leq \mu\left(\lambda^{-1} - \frac{\ep_0}{2}\right)\\
	&\Leftrightarrow& \lambda + \ep_0\left(\frac{\ep_0}{8} + \frac{1}{2} - \frac{1}{4}\lambda^{-1}\right) \leq \frac{\lambda + 2\ep_0}{(\lambda+\ep_0)^{-1} - \ep_0}\left(\lambda^{-1} - \frac{\ep_0}{2}\right).%
\end{eqnarray*}
Using \eqref{eq_step2_ep0} we find%
\begin{eqnarray*}
  && 16\ep_0\left(\frac{\ep_0}{8} + \frac{1}{2} - \frac{1}{4}\lambda^{-1}\right) \leq \ep_0\left[(\lambda^{-1} - 1) + 8 - 4\lambda^{-1}\right] = \ep_0(7-3\lambda^{-1}),\\
  && \Rightarrow \lambda + \ep_0\left(\frac{\ep_0}{8} + \frac{1}{2} - \frac{1}{4}\lambda^{-1}\right) \leq \lambda + \frac{\ep_0}{16}(7-3\lambda^{-1}) \leq \lambda + 2\ep_0.%
\end{eqnarray*}
Thus, it suffices to show that%
\begin{equation*}
  (\lambda+\ep_0)^{-1} - \ep_0 \leq \lambda^{-1} - \frac{\ep_0}{2}.%
\end{equation*}
This is equivalent to $1 \leq (\lambda+\ep_0)(\lambda^{-1} + \ep_0/2)$, which is obviously true. The left-hand side of \eqref{eq_step4_kappamax} is a quadratic polynomial in $\rho$ with highest-order coefficient $\mu\ep_0/4>0$. Thus, the inequality holds for all $\rho\in[0,1]$, concluding Step 4.%

\emph{Step 5.} We claim that $\theta_0(V) \leq A_1^{-1}$ and $|v|\leq r$ implies%
\begin{equation}\label{eq_step5_claim}
  \rme^{-A_1(\theta_0(V) + |v|)} \leq \left|\frac{\det(\rmd\widetilde{\varphi}_{b,x})_v|_V}{\det(\rmd\widetilde{\varphi}_{b,x})_0|_{E^+_{b,x}}}\right|
	\leq \rme^{A_1(\theta_0(V) + |v|)}.%
\end{equation}
Let $e_1,\ldots,e_u$ be an orthonormal frame of $E^+_{b,x}$ and note that $\theta_0(V) \leq A_1^{-1}$ implies $|e_i + L_Ve_i| \leq 1 + A_1^{-1} \leq 2$. Using $A_1/4 = u(2A_0)^u \geq 2^{u-1}u$, we find%
\begin{eqnarray*}
 && \left|(e_1+L_Ve_1) \wedge \cdots \wedge (e_u + L_Ve_u) - e_1 \wedge \cdots \wedge e_u\right|\allowdisplaybreaks\\
  && \leq \left|(e_1+L_Ve_1) \wedge \cdots \wedge (e_u + L_Ve_u) - (e_1 + L_Ve_1) \wedge e_2 \wedge \cdots \wedge e_u\right|\allowdisplaybreaks\\
  && \quad + \left|(e_1 + L_Ve_1) \wedge e_2 \wedge \cdots \wedge e_u - e_1 \wedge \cdots \wedge e_u\right|\allowdisplaybreaks\\
 && \leq 2 \left| (e_2+L_Ve_2)\wedge \cdots \wedge (e_u+L_Ve_u) - e_2 \wedge \cdots \wedge e_u \right| + \|L_V\|\allowdisplaybreaks\\
 && \leq 2(2 \left| (e_3+L_Ve_3)\wedge \cdots \wedge (e_u+L_Ve_u) - e_3 \wedge \cdots \wedge e_u \right| + \|L_V\|) + \|L_V\|\allowdisplaybreaks\\
 && \leq \cdots \leq u 2^{u-1}\|L_V\| \leq \frac{A_1}{4}(\theta_0(V) + |v|).% 
\end{eqnarray*}
Analogously, using \eqref{eq_step1_opnorms}, one shows%
\begin{eqnarray*}
  \Bigl|(\rmd\widetilde{\varphi}_{b,x})_v(e_1 + L_Ve_1) \wedge \!\!\!\!&\cdots&\!\!\!\! \wedge (\rmd\widetilde{\varphi}_{b,x})_v(e_u + L_Ve_u)\\
	&-& (\rmd\widetilde{\varphi}_{b,x})_v(e_1) \wedge \cdots \wedge (\rmd\widetilde{\varphi}_{b,x})_v(e_u)\Bigr|\\
	&\leq& u2^{u-1}\|L_V\|\|(\rmd\widetilde{\varphi}_{b,x})_v\|^u \leq u2^{u-1}\|L_V\|A_0^u \mbox{\quad and}\\
  \bigl|(\rmd\widetilde{\varphi}_{b,x})_v(e_1) \!\!\!\!&\wedge&\!\!\!\! \cdots \wedge (\rmd\widetilde{\varphi}_{b,x})_v(e_u)\\
	&-& (\rmd\widetilde{\varphi}_{b,x})_0(e_1) \wedge \cdots \wedge (\rmd\widetilde{\varphi}_{b,x})_0(e_u)\bigr| \leq uA_0^u|v|.%
\end{eqnarray*}
The latter is shown by induction. With $C:=(\rmd\widetilde{\varphi}_{b,x})_v$, $D:=(\rmd\widetilde{\varphi}_{b,x})_0$ we have%
\begin{eqnarray*}
  |Ce_1 \!\!\!\!\!&\wedge&\!\!\!\!\! \cdots \wedge Ce_u - De_1 \wedge \cdots \wedge De_u|\allowdisplaybreaks\\
  &&\leq |Ce_1\wedge \cdots\wedge Ce_u - Ce_1 \wedge De_2 \wedge \cdots \wedge De_u|\allowdisplaybreaks\\
  &&\quad + |Ce_1 \wedge De_2 \wedge \cdots \wedge De_u - De_1 \wedge \cdots \wedge De_u|\allowdisplaybreaks\\
  &&\leq \|C\| |Ce_2\wedge \cdots \wedge Ce_u - De_2 \wedge \cdots \wedge De_u| + \|C-D\|\|D\|^{u-1}\allowdisplaybreaks\\
  &&\stackrel{\eqref{eq_step1_opnorms},\eqref{eq_step1_lipconstants}}{\leq} A_0 (u-1)A_0^{u-1}|v| + A_0|v|\|D\|^{u-1} \leq u A_0^u|v|.%
\end{eqnarray*}
With these estimates and $uA_0^u = A_1/2^{u+2}$, we find%
\begin{eqnarray*}
  \bigl|\!\!\!\!\!\!\!\!&&\!\!\!\!\!\!\!\!(\rmd\widetilde{\varphi}_{b,x})_v(e_1 + L_Ve_1) \wedge \cdots \wedge (\rmd\widetilde{\varphi}_{b,x})_v(e_u+L_Ve_u)\\
  && - (\rmd\widetilde{\varphi}_{b,x})_0(e_1)\wedge\cdots\wedge(\rmd\widetilde{\varphi}_{b,x})_0(e_u)\bigr|\\
  &\leq& u 2^{u-1}\|L_V\|A_0^u + u A_0^u|v|\\
  &=& \frac{A_1}{2^{u+2}}\left(2^{u-1}\|L_V\| + |v|\right) \leq \frac{A_1}{4}(\theta_0(V) + |v|).%
\end{eqnarray*}
Using these estimates together with the fact that $|\det(\rmd\widetilde{\varphi}_{b,x})_0|_{E^+_{b,x}}| = |(\rmd\widetilde{\varphi}_{b,x})_0(e_1)\wedge\cdots\wedge(\rmd\widetilde{\varphi}_{b,x})_0(e_u)|>1$, we find%
\begin{eqnarray*}
  \frac{|\det(\rmd\widetilde{\varphi}_{b,x})_v|_V|}{|\det(\rmd\widetilde{\varphi}_{b,x})_0|_{E^+_{b,x}}|} &=& \frac{|(\rmd\widetilde{\varphi}_{b,x})_v(e_1+L_Ve_1) \wedge \cdots \wedge (\rmd\widetilde{\varphi}_{b,x})_v(e_u+L_Ve_u)|}{|(e_1+L_Ve_1)\wedge\cdots\wedge(e_u+L_Ve_u)|}\allowdisplaybreaks\\
 && \cdot \frac{1}{|(\rmd\widetilde{\varphi}_{b,x})_0(e_1)\wedge\cdots\wedge (\rmd\widetilde{\varphi}_{b,x})_0(e_u)|}\allowdisplaybreaks\\
 &\leq& \frac{1 + \frac{A_1}{4}(\theta_0(V) + |v|)}{1 - \frac{A_1}{4}(\theta_0(V) + |v|)}\allowdisplaybreaks\\
 &\leq& 1 + A_1(\theta_0(V)+|v|) \leq \rme^{A_1(\theta_0(V)+|v|)}.%
\end{eqnarray*}
Here we use that $(1 + x/4) \leq (1 + x)(1 - x/4)$ for all $x \in [0,2]$ and%
\begin{equation*}
  A_1(\theta_0(V) + |v|) \leq 1 + A_1r \stackrel{\eqref{eq_step2_rcond}}{\leq} \left(1 + \frac{(1-\beta)(1-\beta^2)}{2A_0K^2}\right) \leq 2.%
\end{equation*}
The proof for the lower estimate works analogously.%
	
\emph{Step 6.} We claim that there exists a constant $B\geq1$ such that for all $(b,x)\in\QC$ and $n\geq0$ it holds that%
\begin{equation}\label{eq_step6_claim}
  B^{-1} \leq \left|\frac{\det (\rmd\widetilde{\varphi}^n_{b,x})_v|_V}{\det(\rmd\widetilde{\varphi}^n_{b,x})_0|_{E^+_{b,x}}}\right| \leq B,%
\end{equation}
whenever $v \in D_b(x,n,r)$ and $\theta_0(V) \leq A_1^{-1}$. (In the following, w.l.o.g.~$n\geq2$.) First we prove inductively that for $v\in D_b(x,n,r)$ and $\theta_0(V)\leq A_1^{-1}$%
\begin{equation}\label{eq_step6_ind}
  \theta_i(v,V) \leq \left\{\begin{array}{rl}
	                             \beta^i\theta_0(V) + irA_2\beta^{i-1} & \mbox{for } 1 \leq i \leq [n/2]\\
															 \beta^i\theta_0(V) + \left[n/2\right]rA_2\beta^{i-1} + rA_3\beta^{n-i+1} & \mbox{for } [n/2] \leq i \leq n%
														\end{array}\right.,%
\end{equation}
where $A_2 = 2K^2A_0$ and $A_3 = A_2\sum_{i=0}^{\infty}\beta^{2i}$. The case $i=1$ follows from%
\begin{equation*}
  \theta_1(v,V) \stackrel{\eqref{eq_step4_theta1est}}{\leq} \mu\theta_0(V) + KA_0|v| \leq \beta\theta_0(V) + KA_0r \leq \beta\theta_0(V) + A_2r.%
\end{equation*}
For the induction step, we first show that $\theta_0(V)\leq A_1^{-1}$ and \eqref{eq_step6_ind} imply that $\theta_i(v,V) \leq A_1^{-1}$ for $1\leq i\leq n-1$. For $1 \leq i \leq [n/2]$ this is done as follows.%
\begin{eqnarray*}
  \theta_i(v,V) &\leq& \beta^iA_1^{-1} + irA_2\beta^{i-1} \stackrel{\eqref{eq_step2_rcond}}{\leq} \beta^iA_1^{-1} + i(1-\beta)\frac{1}{2K^2A_0A_1}2K^2A_0\beta^{i-1}\\
	&=& A_1^{-1}\left(\beta^i + i(1-\beta)\beta^{i-1}\right) \leq A_1^{-1}.%
\end{eqnarray*}
The last inequality follows from Lemma \ref{lem_poly}. Similarly, for $[n/2]\leq i \leq n$,%
\begin{equation*}
  \theta_i(v,V) \leq A_1^{-1}\left(\beta^i + \left[\frac{n}{2}\right](1-\beta^2)(1-\beta)\beta^{i-1} + (1-\beta)\beta^{n-i+1}\right) \leq A_1^{-1},%
\end{equation*}
which also follows from Lemma \ref{lem_poly}. Now, for the induction step note that%
\begin{eqnarray*}
  \theta_{i+1}(v,V) &=& \theta_0\left((\rmd\widetilde{\varphi}^{i+1}_{b,x})_vV\right) = \theta_0\left((\rmd\widetilde{\varphi}_{\phi_i(b,x)})_{\widetilde{\varphi}^i_{b,x}(v)}(\rmd\widetilde{\varphi}^i_{b,x})_vV\right)\\
	&=& \theta_1\left(\widetilde{\varphi}^i_{b,x}(v),(\rmd\widetilde{\varphi}^i_{b,x})_vV\right)
	\stackrel{\eqref{eq_step4_theta1est}}{\leq} \mu\theta_0\left((\rmd\widetilde{\varphi}^i_{b,x})_vV\right) + KA_0\left|\widetilde{\varphi}^i_{b,x}(v)\right|\\
	&\leq& \beta\theta_i(v,V) + KA_0\left|\widetilde{\varphi}^i_{b,x}(v)\right|.%
\end{eqnarray*}
Next we have to distinguish two cases. Let us first assume $1 \leq i \leq [n/2]$. Then the induction hypothesis together with \eqref{eq_step3_claim} gives%
\begin{eqnarray*}
  \theta_{i+1}(v,V) &\leq& \beta\left(\beta^i\theta_0(V) + irA_2\beta^{i-1}\right) + KA_0\left|\widetilde{\varphi}^i_{b,x}(v)\right|\\
	                  &\leq& \beta^{i+1}\theta_0(V) + irA_2\beta^i + 2K^2A_0r\alpha^i\\
										&\leq& \beta^{i+1}\theta_0(V) + irA_2\beta^i + rA_2\beta^i\\
										   &=& \beta^{i+1}\theta_0(V) + (i+1)rA_2\beta^i.%
\end{eqnarray*}
This is the desired estimate if $i+1 \leq [n/2]$. Otherwise, $i=[n/2]$ and thus%
\begin{eqnarray*}
  \theta_{i+1}(v,V) &\leq& \beta^{i+1}\theta_0(V) + irA_2\beta^i + rA_2\beta^i\\
	                  &\leq& \beta^{i+1}\theta_0(V) + \left[\frac{n}{2}\right]rA_2\beta^i + rA_3\beta^{n-(i+1)+1}.%
\end{eqnarray*}
Note that the last inequality is equivalent to $(1-\beta^2)\beta^{[n/2]} \leq \beta^{n-[n/2]}$ which is trivially satisfied if $n$ is even and for odd $n$ reads $1-\beta^2 \leq \beta$, which holds by the choice of $\beta$. Now assume $i>[n/2]$. In this case,%
\begin{eqnarray*}
  \theta_{i+1}(v,V) &\stackrel{\eqref{eq_step4_theta1est}}{\leq}& \beta\left(\beta^i\theta_0(V) + \left[\frac{n}{2}\right]rA_2\beta^{i-1} + rA_3\beta^{n-i+1}\right) + KA_0\left|\widetilde{\varphi}^i_{b,x}(v)\right|\\
	&\stackrel{\eqref{eq_step3_claim}}{\leq}& \beta^{i+1}\theta_0(V) + \left[\frac{n}{2}\right]rA_2\beta^i + rA_3\beta^{n-i+2} + A_2r\beta^{n-i}\\
	&=& \beta^{i+1}\theta_0(V) + \left[\frac{n}{2}\right]rA_2\beta^i + r\beta^{n-i}(A_3\beta^2 + A_2)\\
	&=& \beta^{i+1}\theta_0(V) + \left[\frac{n}{2}\right]rA_2\beta^i + rA_3\beta^{n-i}.%
\end{eqnarray*}
This finishes the proof of \eqref{eq_step6_ind}. Now we can prove \eqref{eq_step6_claim}. Using the sum formulas%
\begin{equation*}
  \sum_{i=1}^n \gamma^i = \frac{\gamma - \gamma^{n+1}}{1-\gamma},\qquad \sum_{i=1}^n i \gamma^{i-1} = \frac{n\gamma^{n+1} - (n+1)\gamma^n + 1}{(1-\gamma)^2},%
\end{equation*}
we can sum up the terms on the right-hand side of \eqref{eq_step6_ind} for $i=1,\ldots,n$. If this sum turns out to be uniformly bounded (w.r.t.~$n$), the claim follows from \eqref{eq_step5_claim} and the chain rule. Since we only need an upper bound, we may forget about the constants. For simplicity, only consider the case where $n=2k$ is even:%
\begin{eqnarray*}
  \sum_{i=1}^k(\beta^i + i\beta^{i-1}) = \frac{\beta - \beta^{k+1}}{1-\beta} + \frac{k\beta^{k+1} - (k+1)\beta^k + 1}{(1-\beta)^2}%
\end{eqnarray*}
and%
\begin{eqnarray*}
  \sum_{i=k+1}^{2k}\left(\beta^i + k\beta^{i-1} + \frac{\beta^{2k-i+1}}{1-\beta^2}\right) = \beta^{k+1}\frac{1-\beta^k}{1-\beta} + k \beta^k\frac{1-\beta^k}{1-\beta} + \frac{\beta-\beta^{k+1}}{(1-\beta)(1-\beta^2)}.%
\end{eqnarray*}
We want to show that the sum of these terms is bounded as $k$ goes to infinity. All the terms in which $k$ only appears as an exponent are bounded from below and above, since $\beta\in(0,1)$. Hence, we only have to take care of the terms in which $k$ appears as a coefficient. Summing up these terms gives%
\begin{equation*}
  \frac{k\beta^{k+1} - (k+1)\beta^k + 1}{(1-\beta)^2} + k\beta^k\frac{1-\beta^k}{1-\beta} = \frac{1-\beta^k-k\beta^{2k}+k\beta^{2k+1}}{(1-\beta)^2}.%
\end{equation*} 
We can treat the denominator as a constant and the counter as a polynomial with parameter $k$. This polynomial has value $1$ at $x=0$ and value $0$ at $x=1$. Moreover, its derivative can be written as $kx^{k-1}(x-1)(2kx^k + \sum_{i=0}^k x^i)$, from which we see that it is decreasing on $[0,1]$ and hence bounded independently of $k$. Applying the chain rule to express the determinants in \eqref{eq_step5_claim} as products now implies the existence of the constant $B$.%

\emph{Step 7.} Let $0 < \rho \leq r/2$ and define for each $(b,x)\in\QC$ the set%
\begin{equation*}
  B_{b,x}(\rho) := B^-_{b,x}(\rho) \tm B^+_{b,x}(\rho),%
\end{equation*}
where $B^{\pm}_{b,x}(\rho) = \{v\in E^{\pm}_{b,x} : |v| < \rho\}$. Then we prove the following claim. If $h:B^+_{b,x}(\rho) \rightarrow B^-_{b,x}(\rho)$ is a $\CC^1$-map with $\Lip(h) \leq A_1^{-1}$, then there is a $\CC^1$-map $k:B^+_{\phi_1(b,x)}(\rho) \rightarrow B^-_{\phi_1(b,x)}(\rho)$ with $\Lip(k) \leq A_1^{-1}$ such that%
\begin{equation}\label{eq_step7_claim}
  \left(\widetilde{\varphi}_{b,x}\Graph(h)\right) \cap B_{\phi_1(b,x)}(\rho) = \Graph(k).%
\end{equation}
To this end, consider for given $h$ the map%
\begin{equation*}
  G:B^+_{b,x}(\rho) \rightarrow E^+_{\phi_1(b,x)},\quad G(v) = B_{b,x}v + R_{b,x}(v + h(v))^+.%
\end{equation*}
We show that $G$ is injective. Take $v_1,v_2 \in B^+_{b,x}(\rho)$ with $G(v_1)=G(v_2)$. Then%
\begin{eqnarray*}
 && B_{b,x}(v_1-v_2) = \left[R_{b,x}(v_2 + h(v_2)) - R_{b,x}(v_1 + h(v_1))\right]^+,\\
&& \stackrel{\eqref{eq_step2_lipconstants}}{\Rightarrow}
  \lambda^{-1}|v_1-v_2| \leq \frac{\ep_0}{2}\left(1 + A_1^{-1}\right)|v_1-v_2| \leq \ep_0|v_1-v_2|.%
\end{eqnarray*}
However, $\lambda\ep_0 < (1/2)\lambda(1-\lambda) < 1/2$ implying $|v_1-v_2|=0$. Now, for given $v_0 \in B^+_{\phi_1(b,x)}(\rho)$ define%
\begin{equation*}
  F:B^+_{b,x}(\rho) \rightarrow E^+_{b,x},\quad F(v) := B^{-1}_{b,x}v_0 - B^{-1}_{b,x}(R_{b,x}(v+h(v))^+).%
\end{equation*}
This map actually takes values in $B^+_{b,x}(\rho)$, since $\ep_0<\lambda^{-1}-1$ implies%
\begin{equation*}
  |F(v)| \leq \|B^{-1}_{b,x}\||v_0| + \|B^{-1}_{b,x}\|K\frac{\ep_0}{2K}\underbrace{|v+h(v)|}_{\leq2\rho} \leq \lambda\rho(1+\ep_0) < \rho.%
\end{equation*}
A similar estimate shows that $F$ is a contraction with contraction constant $\lambda\ep_0<1$. Hence, there is a fixed point $v^* = F(v^*)$, which is equivalent to $G(v^*)=v_0$ (note that we can extend $F$ to the closure of $B^+_{b,x}(\rho)$ which is a complete metric space, and this extension still takes values in $B^+_{b,x}(\rho)$). Since $v_0$ was chosen arbitrarily, we have shown that the image of $G$ contains $B^+_{\phi_1(b,x)}(\rho)$. Since $G(v)$ is the $x$-coordinate of $\widetilde{\varphi}_{b,x}(v+h(v))$, we must define $k$ by%
\begin{equation*}
  k(v) := A_{b,x}h(G^{-1}(v)) + R_{b,x}(G^{-1}(v)) + h(G^{-1}(v)))^-.%
\end{equation*}
It is easy to see that $k$ is a $\CC^1$-map with values in $B^-_{\phi_1(b,x)}(\rho)$, satisfying \eqref{eq_step7_claim}. It remains to show that $\Lip(k) \leq A_1^{-1}$. To this end, note that \eqref{eq_step7_claim} implies%
\begin{equation*}
  (\rmd\widetilde{\varphi}_{b,x})_{v+h(v)}(I + (\rmd h)_v) = (\rmd G)_v + (\rmd k)_{G(v)}(\rmd G)_v,%
\end{equation*}
and $V:= (I+(\rmd h)_v)E^+_{b,x}$ is a subspace of $T_xM$ with $V \oplus E^+_{b,x} = T_xM$. Using \eqref{eq_step4_theta1est}, we find%
\begin{equation*}
  \theta_1(v+h(v),V) \leq \mu\left\|(\rmd h)_v\right\| + KA_0|v+h(v)| \leq \mu A_1^{-1} + 2KA_0\rho,%
\end{equation*}
and therefore%
\begin{equation*}
  \theta_1(v+h(v),V) = \left\|(\rmd h)_{G(v)}\right\| \leq \mu A_1^{-1} + KA_0r \stackrel{\eqref{eq_step2_rcond}}{\leq} \mu A_1^{-1} + \frac{KA_0}{A_0A_1K} = A_1^{-1}.%
\end{equation*}
This implies $\Lip(k) \leq A_1^{-1}$.%

\emph{Step 8.} Let $0<\rho\leq r/2$. We claim that there exists a constant $K_{\rho}>0$ such that for all $(b,x)\in\QC$ it holds that%
\begin{equation*}
  K_{\rho}^{-1} \leq m_{x,h}(\Graph(h)) \leq K_{\rho},%
\end{equation*}
if $h:B^+_{b,x}(\rho) \rightarrow B^-_{b,x}(\rho)$ is a $\CC^1$-map with $\Lip(h) \leq A_1^{-1}$, where $m_{x,h}$ denotes the Lebesgue measure on $\Graph(h)$ induced by its inherited Riemannian metric as a submanifold of $T_xM$ (with the Riemannian inner product $\langle\cdot,\cdot\rangle$). We have%
\begin{eqnarray*}
  m_{x,h}(\Graph(h)) &=& \int_{B^+_{b,x}(\rho)}\sqrt{\det[\id + (\rmd h)_v^* (\rmd h)_v]}\rmd v\\
	                   &=& \int_{B^+_{b,x}(\rho)}\prod_i\sqrt{1 + \lambda_i((\rmd h)_v^* (\rmd h)_v)}\rmd v\\
									&\leq& \int_{B^+_{b,x}(\rho)}\left(1 + A_1^{-2}\right)^{u/2} \rmd v = \vol(B^+_{b,x}(\rho)) \left(1 + A_1^{-2}\right)^{u/2},%
\end{eqnarray*}
where $\lambda_i((\rmd h)_v^* (\rmd h)_v)$ are the eigenvalues of $(\rmd h)_v^* (\rmd h)_v$, i.e., the singular values of $(\rmd h)_v$. Taking the maximum over all $(b,x)$ in the compact set $\QC$, we find a uniform upper bound for $m_{x,h}(\Graph(h))$. Analogously,%
\begin{eqnarray*}
  m_{x,h}(\Graph(h)) &=& \int_{B^+_{b,x}(\rho)}\prod_i \sqrt{1 + \lambda_i((\rmd h)_v^* (\rmd h)_v)}\rmd v\\
									&\geq& \int_{B^+_{b,x}(\rho)}\rmd v \geq \min_{(b,x)\in\QC}\vol(B^+_{b,x}(\rho)) > 0,%
\end{eqnarray*}
which gives a uniform lower bound.%

\emph{Step 9.} Again let $0<\rho\leq r/2$. Fix $(b,x)\in\QC$ and $n\geq0$. For a given $v_- \in B^-_{b,x}(\rho)$ define%
\begin{equation*}
  C(v_-,n,\rho) := \left\{v\in T_xM\ :\ \pi^-_{b,x}(v) = v_-,\ \left\|\widetilde{\varphi}^k_{b,x}(v)\right\|_{\phi_k(b,x)} < \rho,\ 0\leq k \leq n\right\},%
\end{equation*}
where $\|\cdot\|_{\phi_1(b,x)}$ is the maximum norm introduced in Step 1. Note that $C(v_-,n,\rho)$ is the graph of a constant function and $C(v_-,n,\rho) \subset D_b(x,n,r)$. By Step 7, $\widetilde{\varphi}^n_{b,x}C(v_-,n,\rho)$ is the graph of a $\CC^1$-map $h_{v_-,n}:B^+_{\phi_n(b,x)}(\rho) \rightarrow B^-_{\phi_n(b,x)}(\rho)$ with $\Lip(h_{v_-,n})\leq A_1^{-1}$. Hence, by Step 8 we obtain%
\begin{eqnarray*}
  K_{\rho}^{-1} &\leq& m_{\varphi(n,x,b),h_{v_-,n}}\left(\widetilde{\varphi}^n_{b,x}C(v_-,n,\rho)\right)\\
	                 &=& \int_{C(v_-,n,\rho)}\left|\det(\rmd\widetilde{\varphi}^n_{b,x})_v|_{E^+_{b,x}}\right|\rmd m_{x,h_{v_-,0}}(v) \leq K_{\rho},%
\end{eqnarray*}
where $h_{v_-,0}:B^+_{b,x}(\rho) \rightarrow B^-_{b,x}(\rho)$ is the constant map $w \mapsto v_-$. This together with \eqref{eq_step6_claim} implies%
\begin{equation*}
  (K_{\rho}B)^{-1} \leq m_{x,h_{v_-,0}}\left(C(v_-,n,\rho)\right)\left|\det(\rmd\varphi_{n,b})_x|_{E^+_{b,x}}\right| \leq K_{\rho}B.%
\end{equation*}
Therefore, writing $E_{b,x}$ for the inner product space $(T_xM,\|\cdot\|_{b,x})$ and defining%
\begin{equation*}
  N_b(x,n,\rho) := \left\{v \in E_{b,x}\ :\ \left\|\widetilde{\varphi}^k_{b,x}(v)\right\|_{\phi_n(b,x)} < \rho,\ k=0,1,\ldots,n\right\},%
\end{equation*}
it holds by Fubini's theorem that%
\begin{equation*}
  (C_{\rho}')^{-1} \leq m_x\left(N_b(x,n,\rho)\right)\left|\det(\rmd\varphi_{n,b})_x|_{E^+_{b,x}}\right| \leq C_{\rho}',%
\end{equation*}
because%
\begin{equation*}
  N_b(x,n,\rho) = \bigcup_{v_- \in B^-_{b,x}(\rho)}C(v_-,n,\rho).%
\end{equation*}
Here $m_x$ denotes the Lebesgue measure on $T_xM$ associated with the Riemannian inner product and $C_{\rho}'$ is a number depending on $\rho$. The assertion of the volume lemma now follows from the observation that for $0<\ep<r/(2K)$ one has%
\begin{equation}\label{eq_step9_bbincl}
  \exp_x\left(N_b\left(x,n,\frac{\ep}{2}\right)\right) \subset B^n_b(x,\ep) \subset \exp_x\left(N_b\left(x,n,\frac{r}{2}\right)\right)%
\end{equation}
for all $(b,x)\in\QC$ and $n\geq0$, and the volume distortion affected by $\exp_x$ is uniformly bounded on uniformly small balls over the compact set $Q$. The first inclusion in \eqref{eq_step9_bbincl} is shown as follows. Assume $\|\widetilde{\varphi}_{b,x}^k(v)\|_{\phi_k(b,x)}<\ep/2$ for $k=0,1,\ldots,n$. Then%
\begin{eqnarray*}
  \varrho\left(\varphi(k,\exp_x(v),b),\varphi(k,x,b)\right) &=& \varrho\left(\exp_{\varphi(k,x,b)}(\widetilde{\varphi}^k_{b,x}(v)),\exp_{\varphi(k,x,b)}(0)\right)\\
  &=& \left|\widetilde{\varphi}^k_{b,x}(v)\right| \leq 2\left\|\widetilde{\varphi}^k_{b,x}\right\|_{\phi_k(b,x)} < \ep.%
\end{eqnarray*}
To see the second inclusion, assume $\varrho(\varphi(k,y,b),\varphi(k,x,b))<\ep$ for $k=0,1,\ldots,n$. Then $v := \exp_x^{-1}(y)$ is defined and%
\begin{equation*}
  \left\|\widetilde{\varphi}^k_{b,x}(v)\right\|_{\phi_k(b,x)} \leq K\left|\widetilde{\varphi}^k_{b,x}(v)\right| = K\varrho(\varphi(k,y,b),\varphi(k,x,b)) < K\ep < \frac{r}{2}.%
\end{equation*}
This completes the proof.%
\end{proof}

\begin{remark}
The proof of the volume lemma is essentially modelled according to the outline given in Liu \cite{Liu}. Several details of the proof are taken from Qian and Zhang \cite{QZh} who prove a volume lemma for hyperbolic sets of non-invertible maps. Note that the assumption $c=1$ can be removed by considering an iterate $\phi^m$ for $m$ large enough instead of $\phi$.%
\end{remark}

\subsection{The Main Result}%

Now we consider the general control-affine system%
\begin{equation*}
  \Sigma^a:\quad \dot{x}(t) = f_0(x(t)) + \sum_{i=1}^m u_i(t)f_i(x(t)),\quad u\in\UC,%
\end{equation*}
where $f_0,f_1,\ldots,f_m$ are $\CC^2$-vector fields and the set of admissible control functions is given by%
\begin{equation*}
  \UC = \left\{u:\R\rightarrow\R^m\ :\ u \mbox{ measurable with } u(t)\in U \mbox{ a.e.}\right\}%
\end{equation*}
with a compact and convex control range $U\subset\R^m$. Recall that the control flow%
\begin{equation*}
  \phi:\R \tm (\UC \tm M) \rightarrow \UC \tm M,\quad (t,(u,x)) \mapsto (\theta_tu,\varphi(t,x,u))%
\end{equation*}
is a continuous skew-product flow with compact base space $\UC$. From the assumptions it follows that $\varphi$ is of class $\CC^2$ in the $x$-component and the first and second derivatives depend continuously on $(t,x,u) \in \R \tm M \tm \UC$ (cf.~\cite[Thm.~1.1]{Ka2}).%

We say that a compact set $Q\subset M$ is \emph{full-time controlled invariant} if for each $x\in Q$ there exists $u\in\UC$ with $\varphi(\R,x,u)\subset Q$. Then the set%
\begin{equation*}
  \QC := \left\{(u,x) \in \UC \tm M\ :\ \varphi(\R,x,u) \subset Q\right\},%
\end{equation*}
called the \emph{full-time lift} of $Q$, is a compact $\phi$-invariant set. We assume that the state space $M$ is endowed with a Riemannian metric.%

\begin{definition}
A compact full-time controlled invariant set $Q \subset M$ is called \emph{uniformly hyperbolic} if for each $(u,x)\in\QC$ there exists a decomposition%
\begin{equation*}
   T_xM = E^-_{u,x} \oplus E^+_{u,x}%
\end{equation*}
satisfying the following properties:%
\begin{enumerate}
\item[(H1)] $(\rmd\varphi_{t,u})_x E^{\pm}_{u,x} = E^{\pm}_{\phi_t(u,x)}$ for all $t\in\R$ and $(u,x)\in\QC$.%
\item[(H2)] There exist constants $c,\lambda>0$ such that for all $(u,x)\in\QC$ we have%
\begin{equation*}
  \left|(\rmd\varphi_{t,u})_x v\right| \leq c^{-1}\rme^{-\lambda t}|v| \mbox{\quad for all\ } t\geq0,\ v\in E^-_{u,x},%
\end{equation*}
and%
\begin{equation*}
  \left|(\rmd\varphi_{t,u})_x v\right| \geq c\rme^{\lambda t}|v| \mbox{\quad for all\ } t\geq0,\ v\in E^+_{u,x}.%
\end{equation*}
\item[(H3)] The linear subspaces $E^{\pm}_{u,x}$ vary continuously with $(u,x)$, i.e., the projections $\pi^{\pm}_{u,x}:T_xM \rightarrow E^{\pm}_{u,x}$ along $E^{\mp}_{u,x}$ depend continuously on $(u,x)$.%
\end{enumerate}
\end{definition}

As for classical hyperbolic sets (of autonomous dynamical systems), it can be shown that (H3) actually follows from (H1) and (H2). In particular, the subspaces $E^{\pm}_{u,x}$ are the fibers of subbundles $E^{\pm}\rightarrow\QC$ of the vector bundle%
\begin{equation*}
  \bigcup_{(u,x)\in\QC}\{u\}\tm T_xM \rightarrow \QC,\quad (u,v) \mapsto (u,\pi_{TM}(v)),%
\end{equation*}
with the base point projection $\pi_{TM}:TM \rightarrow M$. (cf.~\cite[Sec.~6.3]{Ka2}).%

Though formulated for discrete-time skew-product systems, the volume lemma can be applied to $\Sigma^a$ via time-discretization of the control flow. The assumption that the dimensions of the subspaces $E^{\pm}_{u,x}$ are constant over $\QC$ is automatically satisfied if $\QC$ is connected. This is the case, e.g., if $Q$ is a chain control set, because then $\QC$ is a maximal $\phi$-invariant chain transitive set. The assumption that the constant $c$ be equal to $1$ is satisfied if the time step in the discretization is chosen large enough. For the system $\Sigma^a$, the Bowen-metrics are defined by%
\begin{equation*}
  \varrho_{\tau,u}(x,y) := \max_{t\in[0,\tau]}\varrho(\varphi(t,x,u),\varphi(t,y,u)),\quad \tau>0,\ u\in\UC,%
\end{equation*}
and we denote the Bowen-balls of order $\tau>0$ by $B^{\tau}_u(x,\ep)$. Writing%
\begin{equation*}
  J^+((\rmd\varphi_{\tau,u})_x) := \left|\det(\rmd\varphi_{\tau,u})_x|_{E^+_{u,x}}:E^+_{u,x}\rightarrow E^+_{\phi_{\tau}(u,x)}\right|%
\end{equation*}
for the unstable determinant, the volume lemma reads as follows.%

\begin{lemma}
Consider the control-affine system $\Sigma^a$ and assume that $Q\subset M$ is a compact full-time controlled invariant set which is uniformly hyperbolic. If the dimensions of the subspaces $E^{\pm}_{u,x}$ are constant on $\QC$, then for every sufficiently small $\ep>0$ there is $C_{\ep}\geq1$ such that for all $\tau\geq0$ and $(u,x)\in\QC$,%
\begin{equation*}
  C_{\ep}^{-1} \leq \vol\left(B^{\tau}_u(x,\ep)\right) \cdot J^+((\rmd\varphi_{\tau,u})_x) \leq C_{\ep}.%
\end{equation*}
\end{lemma}

In order to prove our main result, we use the property of additive cocycles described in the following proposition.% 

\begin{proposition}{(\cite{KSt})}\label{prop_addcoc}
Let $\Phi:\R\tm X \rightarrow X$ be a continuous flow on a Hausdorff space $X$ and $a:\R \tm X \rightarrow \R$ be a continuous additive cocycle over $\Phi$, i.e., $a(t+s,x) \equiv a(t,x) + a(s,\Phi(t,x))$. Given a compact $\Phi$-invariant set $K\subset X$,%
\begin{equation*}
  \inf_{x\in K}\limsup_{t\rightarrow\infty}\frac{1}{t}a(t,x) = \lim_{t\rightarrow\infty}\inf_{x\in K}\frac{1}{t}a(t,x).%
\end{equation*}
\end{proposition}

\begin{proposition}\label{prop_conj}
Let $Q$ be a compact full-time controlled invariant set of the control-affine system $\Sigma^a$ with full-time lift $\QC$. Further assume that for every $u\in\UC$ there exists at most one $x = x(u) \in Q$ such that $(u,x(u))\in\QC$. Let $\UC_Q \subset \UC$ be the set of all $u\in\UC$ such that $x(u)$ exists. Then the map%
\begin{equation*}
  \sigma:\UC_Q \rightarrow \QC,\quad u \mapsto (u,x(u)),%
\end{equation*}
is a topological conjugacy between the shift flow restricted to $\UC_Q$ and the control flow restricted to $\QC$. If additionally $\UC_Q = \UC$ and $Q$ is a chain control set such that local accessibility holds on $\inner Q\neq\emptyset$, then $Q$ is the closure of a control set.%
\end{proposition}

\begin{proof}
Note that $\UC_Q$ is a compact shift-invariant set, since it is the projection of the compact $\phi$-invariant set $\QC$ to $\UC$. The map $\sigma$ is obviously bijective with continuous inverse $\sigma^{-1}(u,x) = u$. Since both $\UC_Q$ and $\QC$ are compact metric spaces, it is a homeomorphism. The conjugacy identity reads $x(\theta_tu) = \varphi(t,x(u),u)$, which holds by assumption, since both sides of the equation are points, which are kept in $Q$ by the control function $\theta_tu$. For the second assertion, we use that the shift flow is topologically mixing (see \cite[Prop.~4.1.1]{CKl}). Then the first assertion implies that also the control flow on $\QC$ is topologically mixing. Intuitively, this means that the dynamics on $\QC$ is indecomposable, and projecting to $M$ it means that complete approximate controllability holds on the interior of $Q$. Maximality follows from the fact that $Q$ is a chain control set. This is made precise in \cite[Thm.~4.1.3]{CKl}, which immediately implies the assertion.%
\end{proof}

Now we are in position to prove the main theorem of this section.%

\begin{theorem}\label{thm_ie_hyperbolic_lb}
Let $Q$ be a hyperbolic set of the control-affine system $\Sigma^a$ with full-time lift $\QC$ such that the dimensions of the stable and unstable subspaces are constant on $\QC$. Further assume that for every $u\in\UC$ there exists at most one $x = x(u) \in Q$ such that $(u,x(u))\in\QC$. Then for every compact set $K \subset Q$ of positive volume the invariance entropy satisfies%
\begin{equation*}
  h_{\inv}(K,Q) \geq \inf_{(u,x)\in\QC}\limsup_{\tau\rightarrow\infty}\frac{1}{\tau}\log\left|\det(\rmd\varphi_{\tau,u})|_{E^+_{u,x}}\right|.%
\end{equation*}
\end{theorem}

\begin{proof}
The proof is subdivided into three steps.%

\emph{Step 1.} For each $u\in\UC$ and $\tau>0$ we define the set%
\begin{equation*}
  Q^{\pm}(u,\tau) := \left\{ x\in M\ :\ \varphi_{t,u}(x) \in Q,\ \forall t\in[-\tau,\tau] \right\}.%
\end{equation*}
We claim that%
\begin{equation}\label{eq_epmchar}
  Q^{\pm}(u,\tau) = \left\{ x(u^*)\ :\ u^* \in \UC_Q,\ u^*|_{[-\tau,\tau]} = u|_{[-\tau,\tau]} \right\}.%
\end{equation}
Indeed, since $\varphi(t,x,u^*)$ only depends on the restriction of $u^*$ to $[0,t]$ and $\varphi(-t,x,u^*)$ only on the restriction of $u^*$ to $[-t,0]$ for any $t>0$, for all $t\in[-\tau,\tau]$,%
\begin{equation*}
  Q \ni \varphi_{t,u^*}(x(u^*)) = \varphi_{t,u}(x(u^*)) \mbox{\quad for\ } u^* \in \UC_Q \mbox{ with } u^*|_{[-\tau,\tau]} = u|_{[-\tau,\tau]}.%
\end{equation*}
Hence, it follows that%
\begin{equation*}
  \left\{ x(u^*)\ :\ u^*\in\UC_Q,\ u^*|_{[-\tau,\tau]} = u|_{[-\tau,\tau]} \right\} \subset Q^{\pm}(u,\tau).%
\end{equation*}
Conversely, if $x \in Q^{\pm}(u,\tau)$, we can find $u_1,u_2\in\UC$ such that the function%
\begin{equation*}
  u^*(t) := \left\{\begin{array}{ll}
	                      u_1(t) & \mbox{if } t \in (-\infty,-\tau),\\
												  u(t) & \mbox{if } t \in [-\tau,\tau],\\
												u_2(t) & \mbox{if } t \in (\tau,\infty)%
										\end{array}\right.%
\end{equation*}
satisfies $(u^*,x) \in \QC$, implying $x = x(u^*)$. This completes the proof of \eqref{eq_epmchar}.%

\emph{Step 2.} Let $\ep>0$, $u_0\in\UC_Q$. By continuity of $u\mapsto x(u)$ there exist $\delta>0$ and $x_1,\ldots,x_l \in L^1(\R,\R^m)$ with%
\begin{equation*}
  x(V_{2\delta,u_0}) \subset B\left(x(u_0),\frac{\ep}{2}\right),%
\end{equation*}
where%
\begin{equation*}
  V_{2\delta,u_0} := \left\{ u\in\UC_Q\ :\ \left|\int_{\R}\langle u(t)-u_0(t),x_i(t) \rangle \rmd t\right| < 2\delta,\ i=1,\ldots,k \right\}.%
\end{equation*}
Then there is $\tau_0>0$ with%
\begin{equation*}
  \int_{\R \backslash [-\tau_0,\tau_0]}|x_i(t)|\rmd t < \frac{\delta}{\diam U},\quad i=1,\ldots,k.%
\end{equation*}
Now let $u \in V_{\delta,u_0} = \{u\in\UC_Q\ : \ |\int \langle u(t)-u_0(t),x_i(t) \rangle \rmd t| < \delta,\ i=1,\ldots,k\}$ and consider $u^*\in\UC_Q$ with $u^*|_{[-\tau_0,\tau_0]} = u|_{[-\tau_0,\tau_0]}$. We have%
\begin{equation*}
  \left|\int_{\R}\langle u^*(t)-u(t),x_i(t)\rangle \rmd t\right| \leq \diam U \int_{\R\backslash[-\tau_0,\tau_0]}|x_i(t)|\rmd t < \delta,%
\end{equation*}
implying, for $i=1,\ldots,k$,%
\begin{eqnarray*}
  \left|\int_{\R}\langle u^*(t) - u_0(t),x_i(t) \rangle \rmd t\right| &\leq& \left|\int_{\R}\langle u^*(t)-u(t),x_i(t) \rangle \rmd t\right|\\
	&& + \left|\int_{\R} \langle u(t)-u_0(t),x_i(t)\rangle \rmd t\right| < 2\delta.%
\end{eqnarray*}
Hence, $x(u^*) \in B(x(u_0),\ep/2)$. We also have $x(u) \in B(x(u_0),\ep/2)$, which gives $x(u^*) \in B(x(u),\ep)$ and hence%
\begin{equation*}
  Q^{\pm}(u,\tau_0) \subset B(x(u),\ep) \mbox{\quad for all\ } u \in V_{\delta,u_0}.%
\end{equation*}
Letting $u_0$ range through $\UC_Q$, the open sets $V_{\delta(u_0),u_0}$ cover $\UC_Q$. By compactness, we can pick a finite subcover. This implies the existence of $\tau_0>0$ such that%
\begin{equation}\label{eq_uniformconv}
  Q^{\pm}(u,\tau_0) \subset B(x(u),\ep) \mbox{\quad for all\ } u \in \UC_Q.%
\end{equation}

\emph{Step 3.} Now consider the invariance entropy. If $\SC\subset\UC$ is a minimal (w.l.o.g.~finite) $(\tau,K,Q)$-spanning set, then%
\begin{equation}\label{eq_ethetaw}
   K \subset \bigcup_{u\in\SC}Q(u,\tau),\ Q(u,\tau) = \left\{ x\in M\ :\ \varphi([0,\tau],x,u) \subset Q \right\}.%
\end{equation}
We may assume that $\SC \subset \UC_Q$, since, by minimality, for every $u\in\SC$ there exists $x\in Q$ with $\varphi([0,\tau],x,u) \subset Q$ and $u$ can be modified outside of $[0,\tau]$ in such a way that $\varphi(\R,x,u) \subset Q$, implying $u \in \UC_Q$. It is easy to see that%
\begin{equation*}
  \varphi_{\tau,u}(Q(u,2\tau)) = Q^{\pm}(\theta_{\tau}u,\tau).%
\end{equation*}
Consider an arbitrary $T>0$ and $t \in [0,T]$. Then%
\begin{eqnarray*}
  && \varphi_{t,\theta_{\tau}u}\left(\varphi_{\tau,u}(Q(u,2\tau + T))\right)\\
	&& = 
	   \left\{ x \in M\ :\ \varphi(s - t - \tau,x,\theta_{t+\tau}u) \in Q,\ \forall s \in [0,2\tau+T] \right\}\\
		&& = \left\{ x \in M\ :\ \varphi(s,x,\theta_{t+\tau}u) \in Q,\ \forall s \in [-\tau-t,\tau+T-t] \right\} \subset Q^{\pm}(\theta_{t+\tau}u,\tau).%
\end{eqnarray*}
Combining this with \eqref{eq_uniformconv}, we find that for every $\ep>0$ there exists a $\tau>0$ such that for all $u\in\UC_Q$ and $T>0$,%
\begin{equation*}
  \varphi_{\tau,u}(Q(u,2\tau+T)) \subset \bigcap_{t\in[0,T]}\varphi_{t,\theta_{\tau}u}^{-1}B(x(\theta_{t+\tau}u),\ep).%
\end{equation*}
Using the conjugacy $\sigma$ and putting $y := x(\theta_{\tau}u)$, $v := \theta_{\tau}u$, we find that the right-hand side is the Bowen-ball $B^T_v(y,\ep)$. For sufficiently small $\ep$, the volume lemma thus implies%
\begin{equation*}
  \vol(Q(u,2\tau+T)) \leq \vol\left(\varphi_{\tau,u}^{-1}B^T_v(y,\ep)\right) \leq C\left|\det(\rmd \varphi_{T,v})|_{E^+_{v,y}}\right|^{-1},%
\end{equation*}
where we use that $|\det(\rmd\varphi_{\tau,u}^{-1})_x|$ is uniformly bounded on a small neighborhood of the compact set $\QC$. If $\SC_{2\tau+T}$ is a minimal finite $(2\tau+T,K,Q)$-spanning set, then from \eqref{eq_ethetaw} we get%
\begin{equation*}
  0 < \vol(K) \leq r_{\inv}(2\tau+T,K,Q) \max_{u\in\SC_{2\tau+T}}\vol(Q(u,2\tau+T)),%
\end{equation*}
implying%
\begin{eqnarray*}
  h_{\inv}(K,Q) &\geq& \limsup_{T\rightarrow\infty}\frac{1}{2\tau+T}\log\min_{u\in\SC_{2\tau+T}}\left|\det(\rmd\varphi_{T,\theta_{\tau}u})|_{E^+_{\phi_{\tau}(u,x(u))}}\right|\\
	&\geq& \limsup_{T\rightarrow\infty}\inf_{u\in\UC_Q}\frac{1}{T}\log\left|\det(\rmd \varphi_{T,u})|_{E^+_{u,x(u)}}\right|\\
	&=& \inf_{u\in\UC_Q}\limsup_{t\rightarrow\infty}\frac{1}{t}\log\left|\det(\rmd\varphi_{t,u})|_{E^+_{u,x(u)}}\right|\\
	&=& \inf_{(u,x)\in\QC}\limsup_{t\rightarrow\infty}\frac{1}{t}\log\left|\det(\rmd\varphi_{t,u})|_{E^+_{u,x}}\right|,%
\end{eqnarray*}
where we use Proposition \ref{prop_addcoc} to interchange the $\limsup$ and the infimum. (Note that $\alpha_t(u,x) = \log|\det(\rmd\varphi_{t,u})|_{E^+_{u,x}}|$ is a continuous additive cocycle over $\phi$.)%
\end{proof}

\begin{remark}
The assumption that for each $u$ there exists at most one $x(u)$ is in particular satisfied for small (hyperbolic) control sets that arise around hyperbolic equilibria of uncontrolled systems by adding suitable control terms. This corresponds to the well-known structural results about random hyperbolic sets arising by small random perturbations of Axiom A diffeomorphisms (cf.~Liu \cite[Thm.~1.1]{Liu}). The existence of a small control set around an equilibrium follows from Colonius and Kliemann \cite[Cor.~4.1.7]{CKl} under the inner pair condition. We also note that the preceding proof can partially be extended to non-hyperbolic sets under the assumption that the map $u \mapsto Q^{\infty}(u) := \{ x\in Q : \varphi(\R,x,u) \subset Q\}$ is continuous with respect to the Hausdorff metric on the space of compact subsets of $Q$. In this case, we have to deal with escape rates from small neighborhoods of compact sets for non-autonomous dynamical systems (instead of volumes of Bowen-balls). However, even in the autonomous case there is not much known about such escape rates except for the hyperbolic case (see, e.g., Young \cite{You} and Demers, Young \cite{DYo}).%
\end{remark}

\section{Entropy of Hyperbolic Chain Control Sets}\label{sec_form}%

In this section, we prove that the inequality in Theorem \ref{thm_ie_hyperbolic_lb} becomes an equality for a hyperbolic chain control set. One of the main ingredients in the proof is a shadowing lemma for the shift flow, proved in the following subsection.%

\subsection{A Shadowing Lemma for the Shift Flow}%

In the following, we prove a shadowing lemma for the shift flow and show that it reduces the computation of the Lyapunov spectrum of an additive cocycle to the evaluation of the cocycle on periodic points. First we prove a shadowing lemma for discrete-time shifts that is essentially taken from Akin \cite{Aki}.%

Let $(X,d)$ be a compact metric space. On $X^{\Z}$ we consider the product topology, which is compatible with the metric%
\begin{equation*}
  D(\xi,\eta) = \sup_{i\in\Z}\min\left\{d(\xi_i,\eta_i),\frac{1}{|i|}\right\}%
\end{equation*}
with $\min\{a,1/0\} = a$ by convention. One easily sees that%
\begin{equation}\label{eq_metricprop}
  D(\xi,\eta) \leq \ep \quad\Leftrightarrow\quad d(\xi_i,\eta_i) \leq \ep \mbox{ for } |i| < 1/\ep.% 
\end{equation}
We denote by $s$ the shift homeomorphism on $X^{\Z}$, $s((\xi_n)) = (\xi_{n+1})$. An $\ep$-chain of $s$ is a sequence $\{\xi^i : i\in\Z\}$ with $\xi^i\in X^{\Z}$ such that $D(s(\xi^i),\xi^{i+1}) \leq \ep$ for all $i\in\Z$. An orbit $\{s^i(\eta) : i\in\Z\}$ $\ep$-shadows a sequence $\xi^i \in X^{\Z}$ if $D(s^i(\eta),\xi^i) \leq \ep$ for all $i\in\Z$. A periodic $\ep$-chain is an $\ep$-chain $\{\xi^i : i \in \Z\}$ such that $\xi^{i+N} = \xi^i$ for some $N\in\N$ and all $i\in\Z$.%

\begin{proposition}
For every $\ep>0$ there exists $\delta>0$ such that every $\delta$-chain of $s$ is $\ep$-shadowed by an orbit. Moreover, if the $\delta$-chain is periodic, the shadowing orbit is periodic as well (with the same period).%
\end{proposition}

\begin{proof}
For given $\ep\in(0,1)$ let $\delta \in (0,\ep^2)$ and assume that $\{\xi^i : i \in\Z\}$ is a $\delta$-chain for $s$, i.e., $\xi^i \in X^{\Z}$ and $D(\xi^{i+1},s(\xi^i)) \leq \delta$ for all $i\in\Z$. Put $\eta_i := \xi^i_0$ for $i\in\Z$. Then $\eta \in X^{\Z}$ satisfies%
\begin{equation}\label{eq_assertion}
  D(s^i(\eta),\xi^i) \leq \sqrt{\delta} < \ep \mbox{ for all } i\in\Z.%
\end{equation}
This is proved as follows. For $|j| < 1/\delta$, using the triangle inequality, we get%
\begin{eqnarray*}
  d(s^i(\eta)_j,\xi^i_j) &=& d(\eta_{i+j},\xi^i_j) = d(\xi_0^{i+j},\xi^i_j)\\
  &\leq& \sum_kd(\xi^{i+j-k-1}_{k+1},\xi_k^{i+j-k})
  = \sum_k d(s(\xi^{i+j-k-1})_k,\xi^{i+j-k}_k),%
\end{eqnarray*}
where the summation is over $0\leq k < j$ if $j>0$ and over $j\leq k < 0$ if $j < 0$. Because $|k| \leq |j| < 1/\delta$, \eqref{eq_metricprop} implies that each term is bounded by $\delta$, because $\{\xi^i\}$ is a $\delta$-chain. Therefore, $|j|<1/\delta$ implies%
\begin{equation*}
  d(s^i(\eta)_j,\xi^i_j) \leq |j|\delta \mbox{\quad for all } i\in\Z.%
\end{equation*}
Hence, if $|j| \leq 1/\sqrt{\delta} < 1/\delta$, then $d(s^i(\eta)_j,\xi^i_j) \leq \sqrt{\delta}$. Then \eqref{eq_metricprop} implies \eqref{eq_assertion}. If $\xi^{i+N} = \xi^i$, it follows that $\eta_{i+N} = \xi^{i+N}_0 = \xi^i_0 = \eta_i$ for all $i\in\Z$. Hence, the shadowing point $\eta$ is a periodic sequence and therefore its orbit under the shift is periodic.%
\end{proof}

Now we consider the continuous-time shift $\theta:\R\tm\UC\rightarrow\UC$, $(t,u)\mapsto\theta_tu$, on the set $\UC$ of admissible control functions, which is a chain transitive dynamical system on a compact metrizable space. By $d_{\UC}$ we denote a fixed metric on $\UC$, compatible with the weak$^*$-topology. We can identify $\UC$ with a product space $X^{\Z}$, where $X := \{u:[0,1]\rightarrow\R^m\ :\ u \mbox{ is measurable with } u(t)\in U \mbox{ a.e.}\}$, and the bijection%
\begin{equation}\label{eq_weakstarproductid}
  \alpha:\UC \rightarrow X^{\Z},\quad u \mapsto (u_k)_{k\in\Z},\quad u_k := u|_{[k,k+1]}(\cdot + k),%
\end{equation}
is used to identify elements of $\UC$ with sequences. We endow $X$ with the weak$^*$-topology of $L^{\infty}([0,1],\R^m) = L^1([0,1],\R^m)^*$ and $X^{\Z}$ with the corresponding product topology. Then $\alpha$ becomes a homeomorphism. Since $\UC$ and $X^{\Z}$ are compact metric spaces, it suffices to show continuity. Let $(u^{(n)})_{n\geq1}$ be a sequence in $\UC$ converging to some $u\in\UC$. Note that the sequence $\alpha(u^{(n)})$ converges to $\alpha(u)$ iff $\alpha(u^{(n)})_k$ converges to $\alpha(u)_k$ for every $k\in\Z$. But this is clearly the case, since $|\int_{\R}\langle u^{(n)}(t) - u(t),x(t) \rangle \rmd t| \rightarrow 0$ for every $L^1$-function $x$ implies%
\begin{eqnarray*}
  &&\left|\int_0^1 \left\langle u^{(n)}(t+k) - u(t+k),y(t) \right\rangle \rmd t\right| = \left|\int_k^{k+1} \left\langle u^{(n)}(t) - u(t),y(t-k) \right\rangle \rmd t\right|\\
	&& = \left|\int_{\R}\left\langle u^{(n)}(t) - u(t), y(t-k) \cdot \chi_{[k,k+1]}(t) \right\rangle \rmd t \right| \rightarrow 0%
\end{eqnarray*}
for every $L^1$-function $y\in L^1([0,1],\R^m)$. Moreover, $\alpha$ is obviously a topological conjugacy between the time-one-map $\theta_1:\UC \rightarrow \UC$ of the shift flow and the shift homeomorphism $s:X^{\Z} \rightarrow X^{\Z}$. The next corollary immediately follows.%

\begin{corollary}\label{cor_shiftshadowing}
The shift flow $\theta:\R\tm\UC\rightarrow\UC$ satisfies the following shadowing property. For every $\ep>0$ there is $\delta>0$ such that for every bi-infinite sequence $(u^i)_{i\in\Z}$ in $\UC$ with $d_{\UC}(\theta_1 u^i,u^{i+1}) \leq \delta$ for all $i\in\Z$ there exists $u\in\UC$ with $d_{\UC}(\theta_i u,u^i) \leq \ep$ for all $i\in\Z$. If $(u^i)_{i\in\Z}$ is periodic with period $n$, then so is $u$.%
\end{corollary}

Now consider the following more general version of chains. An $(\ep,T)$-chain for $\theta$ is given by $n\in\N$, points $u_0,\ldots,u_n \in \UC$ and times $T_0,\ldots,T_{n-1} \geq T$ such that $d_{\UC}(\theta_{T_i}u_i,u_{i+1}) \leq \ep$ for $i=0,1,\ldots,n-1$. Let $a:\R\tm\UC \rightarrow \R$ be a continuous additive cocycle over $\theta$. Then the Morse spectrum of $a$ is defined as follows. To every $(\ep,T)$-chain $\zeta$ we associate the finite-time Morse exponent%
\begin{equation*}
  \lambda(\zeta) := \frac{1}{T(\zeta)}\sum_{i=0}^{n-1}a(T_i,u_i),\quad T(\zeta) = \sum_{i=0}^{n-1}T_i.%
\end{equation*}
Then the Morse spectrum of $a$ is given by%
\begin{equation*}
  \Lambda_{\Mo}(a) := \bigcap_{T,\ep>0}\cl\left\{\lambda(\zeta) : \zeta \mbox{ is an } (\ep,T)\mbox{-chain}\right\}.% 
\end{equation*}
By San Martin and Seco \cite[Thm.~3.2(2)]{SSe}, the Morse spectrum is the same if one only considers chains with integer times $T_i \in \N$. But then we may assume that $T_i = 1$ for all $i$ by adding trivial jumps. Moreover, by \cite[Lem.~8]{KSt}, periodic chains are sufficient to obtain the Morse spectrum. Hence, we only need to consider chains of the form $u_0,u_1,\ldots,u_n$ with $d_{\UC}(\theta_1 u_i,u_{i+1}) \leq \ep$ for $i=0,1,\ldots,n-1$ and $u_0 = u_n$. We call such chains \emph{regular periodic $\ep$-chains}. On the other hand, the Lyapunov spectrum $\Lambda_{\Ly}(a)$ of $a$ is defined as follows. An $a$-Lyapunov exponent is a limit of the form%
\begin{equation*}
  \lambda(u) := \lim_{t\rightarrow\infty}\frac{1}{t}a(t,u),\quad u\in\UC.%
\end{equation*}
The Lyapunov spectrum $\Lambda_{\Ly}(a)$ is the set of all such limits. By \cite[Thm.~3.2(6)]{SSe}, $\Lambda_{\Mo}(a)$ is a compact interval, containing $\Lambda_{\Ly}(a)$, whose endpoints are Lyapunov exponents. In particular,%
\begin{equation*}
  \min\Lambda_{\Mo}(a) = \min\Lambda_{\Ly}(a).%
\end{equation*}

\begin{proposition}\label{prop_perlyap}
It holds that%
\begin{equation*}
  \min\Lambda_{\Ly}(a) = \inf_{u \mathrm{\ periodic}}\lim_{t\rightarrow\infty}\frac{1}{t}a_t(u).%
\end{equation*}
\end{proposition}

\begin{proof}
The map $a(1,\cdot)$ is uniformly continuous on the compact space $\UC$. Hence, for given $\alpha>0$ there is $\ep>0$ such that $d_{\UC}(u,v) \leq \ep$ implies $|a(1,u)-a(1,v)| \leq \alpha/2$. Let $\lambda := \min\Lambda_{\mathrm{Mo}}(a)$ and choose $\delta = \delta(\ep)$ according to the periodic shadowing property in Corollary \ref{cor_shiftshadowing}. There exists a regular periodic $\delta$-chain $\zeta =(u_0,\ldots,u_n)$ with $|\lambda - \lambda(\zeta)| \leq \alpha/2$. The chain $\zeta$ is $\ep$-shadowed by the orbit of an $n$-periodic $u_*$. Then%
\begin{eqnarray*}
  |\lambda - \lambda(u_*)| \leq |\lambda - \lambda(\zeta)| + |\lambda(\zeta) - \lambda(u_*)| \leq \frac{\alpha}{2} + |\lambda(\zeta) - \lambda(u_*)|,%
\end{eqnarray*}
and we have (using the cocycle property of $a$)%
\begin{eqnarray*}
  |\lambda(\zeta) - \lambda(u_*)| &=& \left|\frac{1}{n}\sum_{i=0}^{n-1}a(1,u_i) - \lim_{t\rightarrow\infty}\frac{1}{t}a(t,u_*)\right|\allowdisplaybreaks\\
  &=& \left|\frac{1}{n}\sum_{i=0}^{n-1}a(1,u_i) - \frac{1}{n}a(n,u_*)\right|\allowdisplaybreaks\\
  &=& \frac{1}{n}\left|\sum_{i=0}^{n-1}(a(1,u_i) - a(1,\theta_i u_*))\right|\allowdisplaybreaks\\
  &\leq& \frac{1}{n}\sum_{i=0}^{n-1}\left|a(1,u_i) - a(1,\theta_i u_*)\right| \leq \frac{\alpha}{2}.%
\end{eqnarray*}
Putting everything together, we end up with $|\lambda - \lambda(u_*)| \leq \alpha$, showing that $\lambda$ can be approximated by Lyapunov exponents of periodic points.%
\end{proof}

\subsection{The Main Result}%

In the following, we consider a control-affine system%
\begin{equation*}
  \Sigma:\quad \dot{x}(t) = f_0(x(t)) + \sum_{i=1}^m u_i(t)f_i(x(t)),\quad u\in\UC.%
\end{equation*}
We assume that $Q$ is a hyperbolic chain control set of $\Sigma$ with nonempty interior. Moreover, we make the following assumptions:%
\begin{enumerate}
\item[(a)] The vector fields $f_0,f_1,\ldots,f_m$ are smooth and the Lie algebra rank condition (for local accessibility) is satisfied on $\inner Q$.%
\item[(b)] For each $u\in\UC$ there exists a unique $x(u)\in Q$ with $(u,x(u))\in\QC$.%
\end{enumerate}
Recall that condition (b) implies that the map $u \mapsto (u,x(u))$ is a topological conjugacy between the shift on $\UC$ and the control flow on $\QC$, and that (a) and (b) together imply that $Q$ is the closure of a control set $D$ (see Prop.~\ref{prop_conj}). For the proof of our main result, we need a series of approximation lemmas.%

\begin{lemma}\label{lem_diff}
The set $\CC^{\infty}(\R,\inner U)$ is dense in $\UC$ w.r.t.~the weak$^*$-topology.%
\end{lemma}

\begin{proof}
By Sontag \cite[Rem.~C.1.2]{Son} we can approximate a given $u_0 \in \UC$ pointwise almost everywhere by $\CC^{\infty}$-functions with values in $\inner U$ on every compact subinterval of $\R$. Let $I_k := [-k-\delta,k+\delta]$, $k\in\N$, for a fixed $\delta>0$. Then there exists for each $k\in\N$ a sequence $(f_n^{(k)})_{n\geq1}$ in $\CC^{\infty}(I_k,\inner U)$ converging almost everywhere to $u_0|_{I_k}$. Using smooth cut-off functions, we may assume that the $f_n^{(k)}$ are defined on $\R$ so that pointwise convergence to $u_0$ (for each $k$) holds on $[-k,k]$. Choosing cut-off functions with values in $[0,1]$, we still have $f_n^{(k)}(t) \in \inner U$ for all $t\in\R$, since $U$ is convex and $0 \in \inner U$. Now take a neighborhood of $u_0$ of the form%
\begin{equation*}
  W = \left\{ u \in \UC \ :\ \left|\int_{\R} \langle u(t) - u_0(t),x_i(t) \rangle \rmd t\right| < 1,\ i=1,\ldots,l \right\},%
\end{equation*}
where $x_1,\ldots,x_l \in L^1(\R,\R^m)$. Then there exists $k\in\N$ such that%
\begin{equation*}
  \int_{\R\backslash [-k,k]}|x_i(t)| \rmd t < \frac{1}{2\diam U},\quad i = 1,\ldots,l.%
\end{equation*}
Consider for each $i \in \{1,\ldots,l\}$ the sequence $(v_n^i)_{n\in\N}$ defined by%
\begin{equation*}
  v_n^i(t) := \langle f_n^{(k)}(t),x_i(t) \rangle,\quad v_n^i \in L^1(\R,\R).%
\end{equation*}
On $[-k,k]$ this sequence converges almost everywhere to $\langle u_0,x_i \rangle$. Moreover,%
\begin{equation*}
  |v_n^i(t)| \leq |f_n^{(k)}(t)| \cdot |x_i(t)| \leq \|f_n^{(k)}\|_{\infty} |x_i(t)|.%
\end{equation*}
Hence, the theorem of dominated convergence yields%
\begin{equation*}
  \left|\int_{[-k,k]} \langle u_0(t) - f_n^{(k)}(t), x_i(t) \rangle \rmd t\right| \leq \int_{[-k,k]}\left|\langle u_0(t),x_i(t) \rangle - v_n^i(t)\right|\rmd t \rightarrow 0.%
\end{equation*}
Consequently, we may choose $n$ large enough so that%
\begin{equation*}
  \left|\int_{[-k,k]} \langle u_0(t) - f_n^{(k)}(t), x_i(t) \rangle \rmd t\right| < \frac{1}{2},\quad i = 1,\ldots,l,%
\end{equation*}
and we can conclude that%
\begin{eqnarray*}
  && \left|\int_{\R} \langle u_0(t) - f_n^{(k)}(t), x_i(t) \rangle \rmd t\right| \leq \left|\int_{[-k,k]}\langle u_0(t) - f_n^{(k)}(t), x_i(t) \rangle \rmd t\right|\\
	&& \qquad + \left|\int_{\R\backslash [-k,k]} \langle u_0(t) - f_n^{(k)}(t), x_i(t) \rangle \rmd t\right| < \frac{1}{2} + \frac{1}{2} = 1.%
\end{eqnarray*}
Hence, $f_n^{(k)} \in W$. Since the family of all such neighborhoods forms a subbasis of the weak$^*$-topology on $\UC$, the proof is complete.%
\end{proof}

\begin{lemma}\label{lem_approx}
Let $u \in L^1(\R,\R^m)$ and let $\sigma_n:\R\rightarrow\R$ be a sequence of diffeomorphisms satisfying the following assumptions:%
\begin{enumerate}
\item[(i)] The sequences $\sigma_n$ and $\sigma_n^{-1}$ converge locally uniformly to the identity.%
\item[(ii)] The sequence of derivatives $(\rmd/\rmd t)\sigma_n^{-1}$ converges locally uniformly to the constant function with value $1$.%
\end{enumerate}
Then for every $\tau>0$ we have%
\begin{equation*}
  \lim_{n\rightarrow\infty}\int_0^{\tau}|u(\sigma_n(t)) - u(t)|\rmd t = 0.%
\end{equation*}
If, additionally, $u$ is essentially bounded and $z\in L^1(\R,\R^m)$, then%
\begin{equation*}
  \lim_{n\rightarrow\infty}\int_0^{\tau} \left|u(\sigma_n(t)) - u(t)\right| |z(t)| \rmd t = 0.%
\end{equation*}
\end{lemma}

\begin{proof}
For a continuous $u$, the first statement is trivial. If $u$ is an arbitrary element of $L^1(\R,\R^m)$, then there exists a sequence $u_k$ of continuous functions with $\int_{\R}|u_k(t) - u(t)|\rmd t \rightarrow 0$. It follows that%
\begin{eqnarray*}
  \int_0^{\tau}|u(\sigma_n(t)) - u(t)|\rmd t &\leq& \int_0^{\tau}|u(\sigma_n(t)) - u_k(\sigma_n(t))|\rmd t\\
	&+& \int_0^{\tau}|u_k(\sigma_n(t)) - u_k(t)|\rmd t + \int_0^{\tau}|u_k(t) - u(t)|\rmd t.%
\end{eqnarray*}
Let $\ep>0$. Choose $k$ large enough that $\int_{\R}|u(s) - u_k(s)|\rmd s < \ep/[3(1+\ep)]$. The first integral can be re-written as%
\begin{equation*}
  \int_{\sigma_n^{-1}(0)}^{\sigma_n^{-1}(\tau)} |u(s) - u_k(s)|\frac{\rmd}{\rmd s} \left[\sigma_n^{-1}(s)\right] \rmd s.%
\end{equation*}
Now choose $n_0$ large enough that $(\rmd/\rmd s)\sigma_n^{-1}(s) \leq 1 + \ep$ for all $n\geq n_0$ and $s$ in an interval of the form $[-\rho,\tau+\rho]$, $\rho>0$. 
Then let $n_1 \geq n_0$ with $\sigma_n^{-1}(0),\sigma_n^{-1}(\tau) \in [-\rho,\tau+\rho]$ for $n\geq n_1$. This implies%
\begin{equation*}
  \int_0^{\tau}|u(\sigma_n(t)) - u_k(\sigma_n(t))|\rmd t \leq (1+\ep)\int_{-\rho}^{\tau+\rho}|u(s) - u_k(s)|\rmd s < \frac{\ep}{3}.%
\end{equation*}
Finally, choose $n_2 \geq n_1$ large enough that $\int_0^{\tau}|u_k(\sigma_n(t)) - u_k(t)|\rmd t < \ep/3$ for $n\geq n_2$. This implies%
\begin{equation*}
  \int_0^{\tau}|u(\sigma_n(t)) - u(t)|\rmd t < \frac{\ep}{3} + \frac{\ep}{3} + \frac{\ep}{3} = \ep,%
\end{equation*}
completing the proof of the first statement. The second statement easily follows from the first one, if $z$ is continuous. For an arbitrary $z\in L^1(\R,\R^m)$, let $z_k$ be continuous with $\int_{\R}|z_k(t) - z(t)|\rmd t \rightarrow 0$. Then%
\begin{eqnarray*}
 && \int_0^{\tau} \left|u(\sigma_n(t)) - u(t)\right| |z(t)| \rmd t \leq 2\|u\|_{\infty}\int_0^{\tau}|z(t)-z_k(t)|\rmd t\\
	&& + \int_0^{\tau}\left|u(\sigma_n(t)) - u(t)\right| |z_k(t)| \rmd t.%
\end{eqnarray*}
Choose $k$ large enough that the first integral is $< \ep/[4\|u\|_{\infty}]$ and then choose $n$ large enough for the second one to become smaller than $\ep/2$.%
\end{proof}

The main approximation lemma needed to prove the desired formula for the entropy of $Q$ reads as follows.%

\begin{lemma}\label{lem_approxper}
Let $(u,x)\in\QC$ be a $\tau$-periodic point of the control flow for some $\tau>0$. Then there exist sequences $\tau_n\rightarrow\tau$ and $u_n\in\UC$ such that each $u_n$ is $\tau_n$-periodic, $(u_n,x(u_n)) \in \inner\UC \tm \inner Q$ and $(u_n,x(u_n)) \rightarrow (u,x)$.%
\end{lemma}

\begin{proof}
The proof is subdivided into five steps.%

\emph{Step 1.} For some $\gamma>1$ consider the time-transformed system%
\begin{equation*}
  \Sigma^{\gamma}:\quad \dot{x}(t) = v(t)\left[f_0(x(t)) + \sum_{i=1}^m u_i(t)f_i(x(t))\right],\quad (v,u) \in \VC^{\gamma} \tm \UC,%
\end{equation*}
with $\VC^{\gamma} := \{v\in L^{\infty}(\R,\R) : v(t) \in [1/\gamma,\gamma] \mbox{ a.e.}\}$. From the proof of Corollary \ref{cor_regforcasys} we know that the trajectories of $\Sigma^{\gamma}$ are time reparametrizations of those of $\Sigma$, more precisely%
\begin{equation*}
  \varphi(\sigma_v(t),x,u) \equiv \varphi^{\gamma}(t,x,(v,u \circ \sigma_v)) \mbox{\ with\ } \sigma_v(t) = \int_0^t v(s)\rmd s.%
\end{equation*}
It follows that $D$ is also a control set of $\Sigma^{\gamma}$. Moreover, we know that $\Sigma^{\gamma}$ satisfies the strong accessibility rank condition on $\inner Q$, since $\Sigma$ satisfies the classical accessibility rank condition on this set.%

\emph{Step 2.} By Corollary \ref{cor_regtrajex}, the set of universally regular control functions for $\Sigma^{\gamma}$ is dense in $\CC^{\infty}(\R,(1/\gamma,\gamma) \tm \inner U)$ with respect to the $\CC^{\infty}$-topology. Let $\mathbf{1}$ denote the constant function with value $1$. Then we find universally regular $(v_n,\bar{u}_n) \in \CC^{\infty}(\R,(1/\gamma,\gamma) \tm \inner U)$ with $v_n \rightarrow \mathbf{1}$ in the $\CC^{\infty}$-topology and $\bar{u}_n \rightarrow u$ in the weak$^*$-topology, since Lemma \ref{lem_diff} guarantees that $u$ can be weakly$^*$-approximated by functions in $\CC^{\infty}(\R,\inner U)$. Now we define the desired sequences as follows. Put%
\begin{equation*}
  \tau_n := \sigma_{v_n}(\tau),\quad u_n(t) := \bar{u}_n \circ \sigma_{v_n}^{-1}(t) \mbox{\ for all } t\in[0,\tau_n],%
\end{equation*}
and extend $u_n$ $\tau_n$-periodically. Since $\tau_n = \int_0^{\tau}v_n(s)\rmd s$ and $v_n \rightarrow \mathbf{1}$ in $\CC^{\infty}$, it follows immediately that $\tau_n \rightarrow \tau$. Since $u_n(\R) = u_n([0,\tau_n]) = \bar{u}_n([0,\tau]) =: K$ and $\bar{u}_n$ is continuous with values in $\inner U$, Lemma \ref{lem_intuchar} implies $u_n \in \inner\UC$.%

\emph{Step 3.} We show $x(u_n) \in \inner Q$. Note that $\varphi(\tau_n,x(u_n),u_n) = x(\theta_{\tau_n}u_n) = x(u_n)$, implying%
\begin{equation*}
  x(u_n) = \varphi(\sigma_{v_n}(\tau),x(u_n),u_n) = \varphi^{\gamma}(\tau,x(u_n),(v_n,\bar{u}_n)).%
\end{equation*}
Since $(v_n,\bar{u}_n)$ is universally regular for $\Sigma^{\gamma}$, the linearization of $\Sigma^{\gamma}$ along the controlled trajectory $(\varphi^{\gamma}(\cdot,x(u_n),(v_n,\bar{u}_n)),(v_n(\cdot),\bar{u}_n(\cdot)))$ is controllable, implying local controllability along this trajectory. Hence, $x(u_n)$ cannot be an element of the boundary $\partial Q = \partial D$, because this would lead to trajectories of $\Sigma^{\gamma}$ that leave $D$ and then return to $D$, in contradiction to the no-return property.%

\emph{Step 4.} We prove that for each $j\in\Z$ the sequence%
\begin{equation*}
  \beta_n(t) := \sigma_{v_n}^{-1}(t + j(\tau-\tau_n)),\quad \beta_n:\R\rightarrow\R,%
\end{equation*}
is a sequence of diffeomorphisms such that both $\beta_n$ and $\beta_n^{-1}$ converge locally uniformly to the identity, and the sequence of derivatives $(\rmd / \rmd t)\beta_n^{-1}$ converges locally uniformly to $\mathbf{1}$. First note that $\beta_n^{-1}(t) = \sigma_{v_n}(t) - j(\tau - \tau_n)$. By definition of $\sigma_{v_n}$ is is clear that $\sigma_{v_n}$ is smooth and invertible. Since the derivative is $v_n(t) \in [1/\gamma,\gamma]$, also the inverse $\sigma_{v_n}^{-1}$ is smooth. Then then same is true for $\beta_n^{-1}$. We have%
\begin{equation*}
  |\beta_n^{-1}(t) - t| \leq |\sigma_{v_n}(t) - t| + j|\tau - \tau_n| \leq \mathrm{sign}(t)\int_0^t|v_n(s) - 1|\rmd s + j|\tau - \tau_n|.%
\end{equation*}
Since $\tau_n\rightarrow\tau$ and $v_n \rightarrow \mathbf{1}$ in $\CC^{\infty}$, it follows that $\beta_n^{-1}$ converges locally uniformly to the identity. Since $(\rmd/\rmd t)\beta_n^{-1}(t) = v_n(t)$, it follows that local uniform convergence also holds for the derivative. The convergence $\beta_n \rightarrow \id$ holds, since $|\beta_n(t) - t| = |s - \beta_n^{-1}(s)|$ for $s = \beta_n^{-1}(t)$ and for every compact interval $[a,b]$, the set $\beta_n^{-1}([a,b])$ is contained in the compact interval $\gamma[a,b] - j(\tau-\tau_n)$.%

\emph{Step 5.} It remains to prove $u_n \rightarrow u$ (which by continuity implies $x(u_n) \rightarrow x(u) = x$). For a given $y\in L^1(\R,\R^m)$ we have to show%
\begin{equation*}
  \int_{\R} \langle u_n(t) - u(t),y(t) \rangle \rmd t \rightarrow 0.%
\end{equation*}
To this end, we first choose $k\in\N$ large enough that $\int_{\R\backslash[-k\tau,k\tau]}|y(t)|\rmd t$ becomes small. Since $\tau_n\rightarrow\tau$, for sufficiently large $n$ the intervals $I_j^n := [j\tau_n,(j+1)\tau_n]$, $j = -(k+1),\ldots,k$, cover $[-k\tau,k\tau]$. Let $I_n := \bigcup_{j=-(k+1)}^kI_j^n$. Then%
\begin{eqnarray*}
 && \left|\int_{I_n} \langle u_n(t) - u(t),y(t) \rangle \rmd t\right| \leq \sum_{j=-(k+1)}^k\left|\int_{I_j^n} \langle u_n(t) - u(t),y(t) \rangle \rmd t\right|\\
	&& \leq (2k+2) \max_j \left|\int_{I_j^n} \langle \bar{u}_n(\sigma_{v_n}^{-1}(t - j\tau_n)) - u(t),y(t) \rangle \rmd t\right|.%
\end{eqnarray*}
In the integral we substitute $t = s + j\tau$. Writing $z_j(s) := y(s+j\tau)$, $\beta_n(s) := \sigma_{v_n}^{-1}(s + j(\tau-\tau_n))$, and using $\tau$-periodicity of $u$, this gives%
\begin{equation*}
  \int_{j(\tau_n-\tau)}^{j(\tau_n-\tau)+\tau_n} \langle \bar{u}_n(\beta_n(s)) - u(s),z_j(s) \rangle \rmd s.%
\end{equation*}
In the bounds between $0$ and $j(\tau_n - \tau)$ and between $\tau$ and $j(\tau_n-\tau)+\tau_n$, the corresponding integral becomes as small as we want for large $n$. Hence, it remains to estimate%
\begin{eqnarray*}
  \left|\int_0^{\tau} \langle \bar{u}_n(\beta_n(s)) - u(s),z_j(s) \rangle \rmd s\right| &\leq& \underbrace{\left|\int_0^{\tau} \langle \bar{u}_n(\beta_n(s)) - u(\beta_n(s)),z_j(s) \rangle \rmd s\right|}_{=: S_1}\\
	&& + \underbrace{\left|\int_0^{\tau} \langle u(\beta_n(s)) - u(s),z_j(s) \rangle \rmd s\right|}_{=: S_2}.%
\end{eqnarray*}
The integral $S_1$ can be estimated as follows.%
\begin{eqnarray*}
  S_1 &\leq& \left|\int_0^{\tau} \langle \bar{u}_n(\beta_n(s)) - u(\beta_n(s)),z_j(s) - z_j(\beta_n(s))\rangle \rmd s\right|\\
	&& + \left|\int_0^{\tau} \langle \bar{u}_n(\beta_n(s)) - u(\beta_n(s)), z_j(\beta_n(s)) \rangle \rmd s\right|\\
	&\leq& \diam U \int_0^{\tau}|z_j(s) - z_j(\beta_n(s))|\rmd s\\
	&& + \left|\int_{\beta_n^{-1}(0)}^{\beta_n^{-1}(\tau)} \langle \bar{u}_n(t) - u(t),z_j(t) \rangle v_n(t) \rmd t\right|.%
\end{eqnarray*}
From Step 4 and Lemma \ref{lem_approx} it follows that the first integral becomes small. For the second one it follows from weak$^*$-convergence $\bar{u}_n \rightarrow u$ and boundedness of $v_n$. To see that $S_2$ becomes small, use Cauchy-Schwarz and Lemma \ref{lem_approx}.%
\end{proof}

Finally, we can prove the the main theorem of this section.%

\begin{theorem}\label{thm_ie_hyperbolic_formula}
Under the assumptions (a) and (b), for every compact set $K\subset D$ of positive volume the invariance entropy satisfies%
\begin{equation*}
  h_{\inv}(K,Q) = \inf_{(u,x)\in\QC}\limsup_{\tau\rightarrow\infty}\frac{1}{\tau}\log\left|\det(\rmd\varphi_{\tau,u})|_{E^+_{u,x}}\right|.%
\end{equation*}
\end{theorem}

\begin{proof}
The lower estimate follows from Theorem \ref{thm_ie_hyperbolic_lb}. Concerning the upper estimate, Proposition \ref{prop_perest} yields%
\begin{equation}\label{eq_perinnest}
  h_{\inv}(K,Q) \leq \inf_{(u,x)}\lim_{t\rightarrow\infty}\frac{1}{t}\log\left|\det(\rmd\varphi_{t,u})|_{E^+_{u,x}}\right|,%
\end{equation}
where the infimum is taken over all periodic $(u,x) \in \inner \UC \tm \inner Q$. Now take an arbitrary periodic $(u,x) \in \QC$ and let $\tau>0$ be its period. By Lemma \ref{lem_approxper}, there exist sequences $u_n \in \inner\UC$ and $\tau_n \rightarrow \tau$ such that each $u_n$ is $\tau_n$-periodic, $x(u_n) \in \inner Q$ and $(u_n,x(u_n)) \rightarrow (u,x)$. Then we have%
\begin{equation*}
  \frac{1}{\tau_n}\log\left|\det(\rmd\varphi_{\tau_n,u_n})|_{E^+_{u_n,x(u_n)}}\right| \rightarrow \frac{1}{\tau}\log\left|\det(\rmd\varphi_{\tau,u})|_{E^+_{u,x(u)}}\right|,%
\end{equation*}
because both $(t,u,x) \mapsto (\rmd\varphi_{t,u})_x$ and $(u,x) \mapsto E^+_{u,x}$ are continuous. This implies that the estimate \eqref{eq_perinnest} also holds, when the infimum is taken over all periodic $(u,x)\in\QC$. Note that $\alpha_t(u,x) := \log|\det(\rmd\varphi_{t,u})|_{E^+_{u,x}}|$ is a continuous additive cocycle over the control flow on $\QC$. Using the topological conjugacy between the control flow on $\QC$ and the shift flow on $\UC$, Proposition \ref{prop_perlyap} implies that $h_{\inv}(K,Q)$ is bounded by the infimum of the full Lyapunov spectrum of $\alpha$, concluding the proof of the upper estimate. To see that the upper limits $\limsup(1/t)\alpha_t(u,x)$ are not smaller than the infimum of the Lyapunov spectrum, see \cite[Cor.~2]{KSt}.%
\end{proof}

\begin{remark}
Finally, we note that the assumption that for each $u$ there exists a unique $x(u)$ with $(u,x(u))\in\QC$ is satisfied in the following two cases: (i) small control sets that arise around hyperbolic equilibria, and (ii) hyperbolic chain control sets of right-invariant systems on flag manifolds (see \cite{DSK}). These and more examples will be discussed in another paper.%
\end{remark}

\section{Acknowledgements}%

We thank Katrin Gelfert and Maxence Novel for very helpful comments, leading to the improvement of the general upper bound in Section \ref{sec_gub} in the first case, and to the proof of Theorem \ref{thm_ie_hyperbolic_lb} in the second case. Moreover, we warmly thank Anne Gr\"{u}nzig for proof-reading the manuscript. The first author was supported by FAPESP scholarship 2013/19756-8 and partially by CAPES grant no.~4229/10-0 and CNPq grant no.~142082/2013-9. The second author was supported by DFG fellowship KA 3893/1-1 and a grant of the Max-Planck-Institut in Bonn, where part of this work was done.%

\end{document}